\documentclass{amsart}

\usepackage{epsfig,amsmath,xypic}
\usepackage[psamsfonts]{amssymb}
\usepackage[latin1]{inputenc}
\usepackage{pstricks,pst-node,pst-plot}
\usepackage{hyperref}

\newcommand{\ob}{{\hspace{.2em}[\hspace{-.4em}[\hspace{.2em}}}
\newcommand{\cb}{{\hspace{.2em}]\hspace{-.4em}]\hspace{.2em}}}

\newtheorem{proposition}{\bf Proposition}

\newtheorem{theorem}{\bf Theorem}
\newtheorem{lemma}{\bf Lemma}
\newtheorem{corollary}{\bf Corollary}

\theoremstyle{remark}

\newtheorem*{remark}{\bf Remark}
\newtheorem*{example}{\bf Example}
\newtheorem{question}{\bf Question}

\def\cal{\mathcal}
\def\sm{{\smallsetminus}}
\def\and{{\quad\text{and}\quad}}
\def\ds{\displaystyle}

\def\N{{\mathbb N}}
\def\Z{{\mathbb Z}}
\def\Q{{\mathbb Q}}
\def\Qbar{{\overline \Q}}
\def\U{{\mathbb U}}
\def\C{{\mathbb C}}
\def\Chat{{\widehat \C}}
\def\Disk{{\mathbb D}}

\def\bE{{\mathcal T}}

\def\H{{\mathcal H}}
\def\P{{\mathbb P}}

\def\M{{\cal M}}
\def\MM{{\cal N}}

\def\spec{{\Sigma}}

\def\cf{{\cal C}_f}
\def\vf{{\cal V}_f}
\def\pf{{\cal P}_f}

\def\Rat{{\rm Rat}}

\def\cQ{{\cal Q}}
\def\qf{{\cal Q}_f}
\def\q{{\bf q}}
\def\btheta{{\boldsymbol\theta}}
\def\bxi{{\boldsymbol\xi}}
\def\btau{{\boldsymbol\tau}}
\def\bvtheta{{\boldsymbol\vartheta}}
\def\bomega{{\boldsymbol\omega}}

\def\m{{\mathfrak m}}

\def\D{{\mathrm D}}
\def\T{{\rm T}}
\def\d{{\rm d}}
\def\dz{{{\rm d}z}}

\def\result{{\rm resultant}}

\def\deg{{\rm deg}}
\def\card{{\rm card}}
\def\ord{{\rm ord}}
\def\id{{\rm id}}
\def\res{{\rm residue}}
\def\Res{{\rm Res}}

\title{Eigenvalues of the Thurston Operator}
\begin{author}[X.~Buff]{Xavier Buff}\thanks{The research of the first author was supported in part by the ANR grant Lambda ANR-13-BS01-0002, the IUF, ICERM, and the Clay Mathematics Institute}
\email{xavier.buff@math.univ-toulouse.fr}
\address{ %
  Institut de Math\'ematiques de Toulouse\\
 Universit\'e Paul Sabatier\\
  118, route de Narbonne \\
  31062 Toulouse Cedex \\
  France }
\end{author}
\begin{author}[A.L.~Epstein]{Adam L. Epstein}
\email{adame@maths.warwick.ac.uk}
\address{ %
Mathematics institute\\
 University of Warwick\\
 Coventry CV4 7AL\\
 United Kingdom }
\end{author}

\begin{author}[S.~Koch]{Sarah Koch}\thanks{The research of the third author was supported in part by the NSF and the Sloan Foundation}
\email{kochsc@umich.edu}
\address{Department of Mathematics\\
530 Church Street\\
East Hall\\
University of Michigan\\
Ann Arbor MI 48109\\
United States }
\end{author}

\begin{document}

\dedicatory{for Bill}

\begin{abstract}
Let $f:\Chat\to\Chat$ be a postcritically finite rational map. Let $\cQ(\Chat)$ be the space of meromorphic quadratic differentials on $\Chat$ with simple poles. We study the set of eigenvalues of the pushforward operator $f_*:\cQ(\Chat)\to \cQ(\Chat)$. In particular, we show that when $f:\C\to \C$ is a unicritical polynomial of degree $D$ with periodic critical point, the eigenvalues of $f_*:\cQ(\Chat)\to \cQ(\Chat)$ are contained in the annulus $\bigl\{\frac{1}{4D}<|\lambda|<1\bigr\}$ and belong to $\frac{1}{D} \U$ where $\U$ is the group of algebraic units.
\end{abstract}

\maketitle

\begin{figure}[htbp]
\centerline{
\includegraphics[height=11cm]{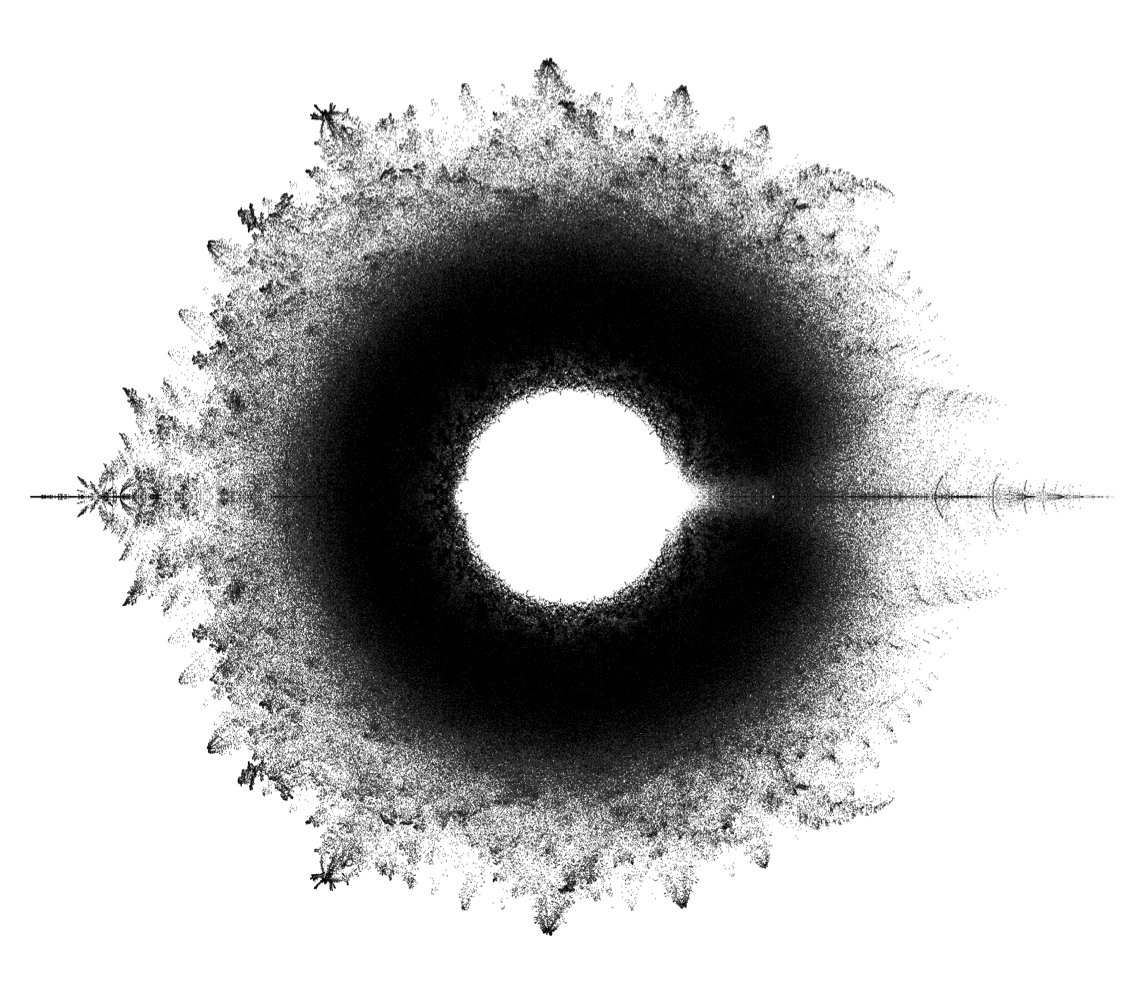}
}
\end{figure}

\section*{Introduction}\label{intro}

\noindent Let $\Chat$ denote the Riemann sphere. Throughout this article, $D\geq 2$ and $\Rat_D$ is the space of rational maps $f:\Chat\to\Chat$ of degree $D$. We denote the set of critical points of $f$ by $\cf$ and the set of critical values by $\vf$. 
The postcritical set $\pf$ is the smallest forward invariant subset of $\Chat$ which contains $\vf$:
\[\pf:=\bigcup_{n\geq 1} f^{\circ n}(\cf).\]

We study {\em postcritically finite rational maps}; that is, rational maps $f:\Chat\to \Chat$ for which $\pf$ is finite. 
It follows from work of Thurston that with the exception of flexible Lattès maps (see \S \ref{sec:lattes} for the definition), postcritically finite rational maps are rigid: if two postcritically finite rational maps are topologically conjugate, then either they are flexible Lattès maps, or they are conjugate by a Möbius transformation \cite{dh}. 

A more elementary result concerns infinitesimal rigidity: if $t\mapsto f_t$ is an analytic family of postcritically finite rational maps, then either the maps are flexible Lattès maps, or there is an analytic family of Möbius transformations $t\mapsto M_t$ such that $M_0={\rm id}$ and $f_t\circ M_t = M_t\circ f_0$. The proof of this result relies on the following lemma, in which $\cQ(\Chat)$ is the space of meromorphic quadratic differentials on $\Chat$ with simple poles and $f_*:\cQ(\Chat)\to \cQ(\Chat)$ is the Thurston pushforward operator (see \S \ref{QDs} for the definition). 

\begin{lemma}[Thurston]
Assume $f\in \Rat_D$ is postcritically finite. If $\lambda$ is an eigenvalue of $f_*:\cQ(\Chat)\to \cQ(\Chat)$, then $|\lambda|\leq 1$.  In addition, $\lambda= 1$ is an eigenvalue if and only if $f$ is a flexible Lattès map. 
\end{lemma}

We recall the proof in \S\ref{sec:contraction} and derive infinitesimal rigidity in \S\ref{sec:rigidity}. We then study the eigenvalues of $f_*:\cQ(\Chat)\to \cQ(\Chat)$. The subspace $\qf\subset \cQ(\Chat)$ of quadratic differentials with poles contained in $\pf$ is invariant by $f_*$.
Let  $\spec_f$ be the set of eigenvalues of $f_*:\qf\to \qf$, and let $\Lambda_f$ be the set of eigenvalues of the induced operator $f_*:\cQ(\Chat)/\qf\to \cQ(\Chat)/\qf$. In \S\ref{sec:qchat}, we study $\Lambda_f$ and in \S\ref{sec:eigenqf}, we study $\spec_f$, establishing the following results. 

\begin{theorem}
The set $\Lambda_f$ consists of $0$ and the complex numbers $\lambda\in \C\sm\{0\}$ such that $1/\lambda^m$ is the multiplier of a cycle of $f$ of period $m$ which is not contained in $\pf$. If $\lambda\in \Lambda_f\sm\{0\}$, then $\lambda$ is an algebraic number but not an algebraic integer. 
\end{theorem}

\begin{theorem}
If $\lambda\in \spec_f$, then $\lambda$ is an algebraic number. If $\lambda$ is an algebraic integer, then either $\lambda=0$, or $f$ is a Lattès map and 
\[\lambda\in \left\{±1,\ ±{\rm i},\  \frac{1}{2}±{\rm i}\frac{\sqrt3}{2},\ -\frac{1}{2}±{\rm i}\frac{\sqrt3}{2}\right\}.\]
\end{theorem}

In \S\ref{sec:charac}, we describe a way to compute the characteristic polynomial of the operator $f_*:\qf\to \qf$. We then apply this in \S\ref{sec:unicritical} to the case where $f$ is a unicritical polynomial with periodic critical point to obtain the following. 

\begin{figure}[htbp]
\centerline{
\includegraphics[width=9cm]{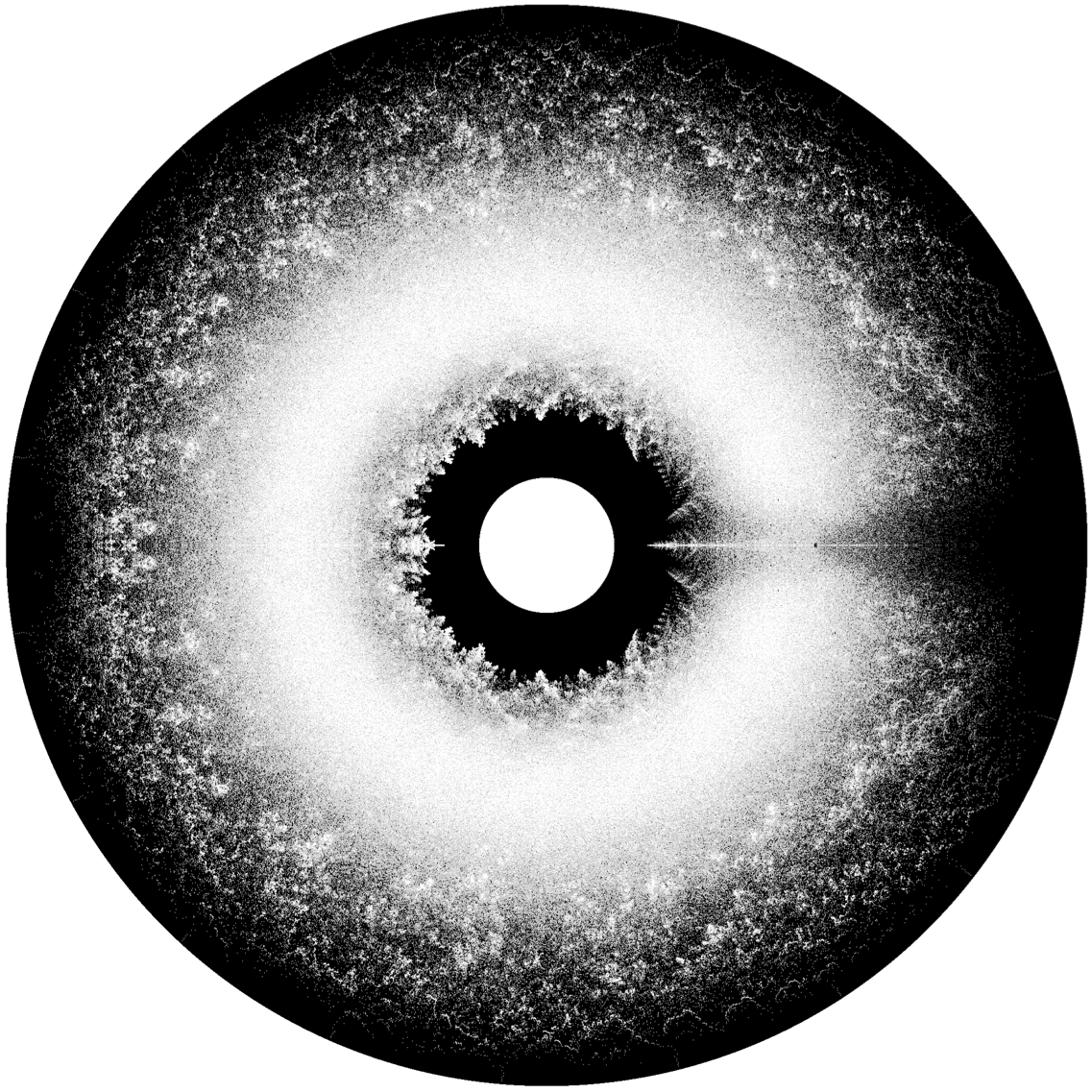}
}
\caption{The set of eigenvalues of $f_*:\qf\to \qf$ for unicritical polynomials $f$ of degree $2$ with periodic critical point, up to period $19$. The set of eigenvalues (white) is contained in the annulus $1/8<|\lambda|<1$ (black). The picture on the first page of this article shows the reciprocal of the eigenvalues (black).\label{fig:bound}}
\end{figure}

\begin{theorem}\label{theo:3}
Let $f:\C\to \C$ be a unicritical polynomial of degree $D$ with periodic critical point. 
If $\lambda$ is an eigenvalue of $f_*:\mathcal Q(\Chat)\to \mathcal Q(\Chat)$, then $D\lambda$ is an algebraic unit and $\frac{1}{4D}<|\lambda|<1$. 
\end{theorem}

Given $D\geq 2$, let $\spec(D)$ be the set of $\lambda\in \C$ which are eigenvalues of $f_*:\qf\to \qf$ for some unicritical polynomial $f:\C\to \C$ of degree $D$ with periodic critical point. 

\begin{theorem}\label{theo:4}
The closure of $\spec(D)$ contains the annulus $r_D\leq |\lambda|\leq 1$ where $r_D$ is defined by 
\[\frac{1}{r_D}= \begin{cases}
2D&\text{if }D\text{ is even}\\
\ds 2D\cos\left(\frac{\pi}{2D}\right)&\text{if }D\text{ is odd}.
\end{cases}
\]
\end{theorem}

The proof of Theorem \ref{theo:3} is given in \S\ref{sec:units} and  \S\ref{sec:gap}, and the proof of Theorem \ref{theo:4} is given in \S\ref{sec:quidistribution}. Finally, in \S\ref{sec:questions}, we pose some questions for further study. 

\medskip
\noindent {\bf Acknowledgments.} We would like to thank Ben Hutz, Curtis McMullen, and Bjorn Poonen for helpful conversations related to this project, and we would like to especially thank Giulio Tiozzo for many exciting discussions and pictures which helped inspire this research. And finally, we thank ICERM for their hospitality when much of this work was completed.

\section{Postcritically finite rational maps}

\noindent Fix $f\in\Rat_D$. It follows from the Riemann-Hurwitz formula that $f$ has $2D-2$ critical points counted with multiplicity, and that $f$ has at least two distinct critical values. As a consequence, $\card(\pf)\geq 2$. If $\card(\pf)=2$, then $\vf=\pf$, and $f$ is conjugate to $z\mapsto z^{\pm D}$. 
 
\subsection{Examples}
In the following two examples, $\card(\pf)=3$.  
\begin{itemize}
\item The polynomial $f:z\mapsto  1-z^D$ has critical set $\cf=\{0,\infty\}$, postcritical set $\pf=\{0,1,\infty\}$, and postcritical dynamics: 
\[
\xymatrix{\infty\ar@(ur,dr)^D}
\quad \xymatrix{0\ar@/^0.7pc/[r]^D &1\ar@/^0.7pc/[l]}
\]

\medskip

\item The rational map $f: z\mapsto 1-1/z^D$ has critical set $\cf=\{0,\infty\}$, postcritical set $\pf=\{0,1,\infty\}$, and postcritical dynamics: 
\[
\xymatrix{0\ar@/^0.7pc/[r]^D & 1\ar@/^0.7pc/[r] & \infty \ar@/^1.1pc/[ll]^D}
\]
\end{itemize}

\noindent {\bf Unicritical polynomials.} For much of this article, we will focus on polynomials $\C\to \C$ of degree $D$ which have a unique critical point; these polynomials are called {\em unicritical}. Every unicritical polynomial is affine conjugate to a polynomial of the form $f_c(z) = z^D+c$, where $z_0=0$ is the unique critical point with critical value $f_c(0)=c$. Fix an integer $m\geq 1$. In parameter space, the roots of the polynomial $G_m(c) := f_c^{\circ m}(0)$ correspond to polynomials $f_c$ for which $0$ is periodic of period dividing $m$. These maps $f_c$ are necessarily postcritically finite with postcritical set equal to 
\[\{\infty\}\;\; \cup \bigcup_{1\leq k\leq m} f_c^{\circ k}(0).\]
An argument due to Gleason shows that $G_m$ has simple roots (see Lemma \ref{lem:gleason}), so there are lots of postcritically finite polynomials. In fact, for $D=2$ the boundary of the Mandelbrot set (see Figure \ref{fig:mandels}) is contained in the closure of the set $\ds\bigcup_{m\geq 1} \{\text {roots of }G_m\}$.

\subsection{Cycles are superattracting or repelling}

Recall that the multiplier of a periodic $m$-cycle $\{x,f(x),\ldots,f^{\circ (m-1)}(x)\bigr\}$ is the eigenvalue $\lambda$ of the linear map $\D_xf^{\circ m}:\T_x\Chat\to \T_x\Chat$ (this eigenvalue does not depend on the point in the cycle). The periodic cycles of a postcritically finite rational map are either {\em superattracting}; that is, $\lambda=0$, or {\em repelling}; that is, $|\lambda|>1$. 

\subsection{Lattès maps\label{sec:lattes}}

The rational map $f:\Chat\to \Chat$ is a {\em{Latt\`es map}} if there is:
\begin{itemize}
\item a complex torus $\bE:=\C/\Lambda$, where $\Lambda\subset \C$ is a lattice of rank 2, 
\item an affine map $L:\bE\to\bE$, and
\item a finite branched cover $\Theta:\bE\to \Chat$
\end{itemize}
so that the following diagram commutes.
\[
\xymatrix{
&\bE\ar[r]^L\ar[d]_\Theta &\bE\ar[d]^\Theta\\
&\Chat\ar[r]_f &\Chat}
\]
The rational map $f$ is necessarily postcritically finite, and $\pf$ is the set of critical values of $\Theta:\bE\to \Chat$. 
In addition, $\card(\pf)\in\{3,4\}$. We shall use the following characterization of Lattès maps with four postcritical points (see \cite[\S4]{jacklattes}). 

\begin{proposition}\label{prop:caraclattes}
A postcritically finite rational map with $\card(\pf)=4$ is a Lattès map if and only if every critical point if simple (with local degree 2) and no critical point is postcritical. 
\end{proposition}

Latt\`es maps are either {\em{flexible}} or {\em{rigid}}. The map $f$ is flexible if 
\begin{itemize}
\item for $L$ of the form $L:w\mapsto \alpha w+\beta $, we have $\alpha \in \Z$, and 
\item the map $\Theta$ has degree $2$. 
\end{itemize}
Equivalently, $f$ is flexible if it can be {\em{deformed}}, that is, if it is part of a one-parameter isospectral family that is nontrivial \cite{ctm}. The Latt\`es map $f$ is rigid if it is not flexible. Flexible Latt\`es maps have four postcritical points (this follows from the fact that $\Theta$ has degree $2$). Rigid Latt\`es maps may have three or four postcritical points. 

\begin{example}\label{rigidlattes}
The map $f:\Chat\to \Chat$ given by $\ds f:z\mapsto {(1-2/z)^2}$ is a rigid Latt\`es map with $\pf=\{0,1,\infty\}$ and has the following postcritical dynamics. 
\[
 \xymatrix{2\ar[r]^2 & 0\ar[r]^2 & \infty \ar[r] &1\ar@(ur,dr)}
\]
\end{example}
\begin{example}\label{flexlattes}
The family 
\[
\{f_t:\Chat\to\Chat\}_{t\in \C\sm \{0,1\}}\quad\text{given by}\quad  f_t:z\mapsto \frac{(z^2-t)^2}{4z(z-1)(z-t)}
\]
consists entirely of flexible Latt\`es maps. The postcritical set of $f_t$ is $\{0,1,t,\infty\}$, and $f_t$ has the following postcritical dynamics. 
\[
\xymatrix{
&\ast\ar[rd]^2 & &\ast\ar[ld]_2\\
\ast\ar[rd]^2 & &0\ar[d] & &\ast\ar[ld]_2\\
&1\ar[r] &\infty\ar@(dl,dr) &t\ar[l]\\
\ast\ar[ru]^2 & & & &  \ast\ar[lu]_2}
\]

\end{example}

\section{Quadratic differentials\label{QDs}}

\noindent Let $U$ be a Riemann surface. A  {\em  quadratic  differential} on $U$ is a section of the square of the cotangent bundle $\T^*U\otimes \T^*U$. We shall usually think of a quadratic differential $\q$ as a field of quadratic forms. In particular, if $\btheta$ is a vector field on $U$ and $\phi$ is a function on $U$, then $\q(\btheta)$ is a function on $U$ and $\q(\phi\btheta) = \phi^2\q(\btheta)$. 


If $\zeta:U\to \C$ is a coordinate, we shall use the notation $(\d\zeta)^2 = \d\zeta\otimes \d\zeta$ (not be confused with $1$-form $\d(\zeta^2)$).  
On $U$ (whose complex dimension is $1$), the  ratio of two quadratic differentials is a function.  In other words, any quadratic differential $\q$ on $U$ may be written as 
\[\q  = q\ (\d\zeta)^2\quad\text{for some function }q.\]

\subsection{Meromorphic quadratic differentials}
A quadratic differential $\q$ on $\Chat$ is meromorphic if $\q=q \ (\dz)^2$ for some meromorphic function $q:\Chat\to \Chat$. 
The quadratic differential $(\dz)^2$ has no zero and has a pole of order $4$ at $\infty$. Since the number of zeros of the function $q$ equals the number of poles of $q$, 
counting multiplicities, the number of poles minus the number of zeros of $\q$ is equal to four. In particular, $\q$  has at least four poles (counting multiplicities). 

Let  $ \cQ (\Chat)$  be the set of
meromorphic quadratic differentials with only simple poles.  
For $X\subset \Chat$, let $\cQ(\Chat;X)\subset \cQ(\Chat)$ be the subset of quadratic differentials whose poles are contained in $X$.  
For $k\geq 0$, let $\cQ_k(\C)$ be the set of meromorphic quadratic differentials whose poles in $\C$ are all simple and which have at worst a pole of order $k$ at $\infty$. 

\begin{example} The quadratic differential $(\dz)^2$ belongs to $\cQ_4(\C)$ and for any $x\in \C$, the quadratic differential 
$\ds \frac{(\dz)^2}{z-x}$ belongs to $\cQ_3(\C)\subset \cQ_4(\C)$. 
\end{example}

\subsection{Pullback}
The derivative $\D f : \T U \to \T V$ of a holomorphic map $f:U\to V$ naturally induces a pullback map $f^*$ from quadratic differentials on $V$ to quadratic differentials on $U$: 
\[f^*\q := \q\circ {\mathrm D}f.\]

\begin{lemma}\label{ordlemmapull}
Let $f:(U,x)\to (V,y)$ be holomorphic  at $x$, and let $q$ be meromorphic at $y=f(x)$. We have  
\[2+\ord_x (f^*\q) = \deg_x f\cdot (2+\ord_y\q).\]
\end{lemma}

\begin{proof}
Let  $\zeta_x:(U,x)\to (\C,0)$  and $\zeta_y:(V,y)\to (\C,0)$ be local coordinates vanishing at respectively $x$ and $y$.  
Set $\delta:=\deg_x f$ so that 
\[\zeta_y \circ f = \zeta_x^\delta\cdot \phi\]
for some function $\phi:U\to \C$ which is holomorphic and does not vanish at $x$. 
Then, 
\[ f^*\left(\frac{\d\zeta_y}{\zeta_y}\right) = \frac{\d(\zeta_y\circ f)}{\zeta_y\circ f} = \delta\frac{\d\zeta_x}{\zeta_x} + \frac{\d \phi}{\phi}.\]
Since $\d\phi/\phi$ is holomorphic near $x$, $f^*(\d\zeta_y/\zeta_y) $ has a simple pole at $x$. It follows that 
\[\ord_x  f^*\bigl((\d\zeta_y)^2/\zeta_y^2\bigr) = \ord_x \bigl(f^*(\d\zeta_y/\zeta_y)\bigr)^2 = 2\,\ord_x f^*(\d\zeta_y/\zeta_y) = -2.\]
Now, if $\q = \psi\cdot (\d\zeta_y)^2/\zeta_y^2$, then $f^*\q = (\psi\circ f)\cdot f^*\bigl((\d\zeta_y)^2/\zeta_y^2\bigr)$, so that 
\[2+\ord _x f^*\q = \ord_x(\psi\circ f) = \deg_x f\cdot \ord_y \psi =  \deg_x f\cdot (2+\ord_y\q).\qedhere\]
\end{proof}

\subsection{The Thurston pushforward operator}

If $f:U\to V$ is a covering map and $\q$ is a quadratic differential on $U$, then we can define a quadratic differential $f_*\q$ on $V$ by
\[f_*\q := \sum_{g \text{ inverse branch of }f} g^* \q.\]
If  $ f: \Chat \to \Chat $ is a nonconstant rational map and $\q=q\ (\dz)^2$ is a meromorphic quadratic differential on $\Chat$,  then the quadratic differential $f_*\q$, which is a priori defined on $\Chat\sm\vf$, is globally meromorphic on $\Chat$, and 
\[ f_* \q: = r\ (\dz)^ 2\quad \text{with}\quad  r (y) := \sum_{x \in f^{-1}(y)} \frac{q (x)}{f '(x) ^ 2}.\]

\begin{lemma}\label{ordlemmapush}
Let $\q$ be a meromorphic quadratic differential on $\Chat$, and let $f:\Chat\to\Chat$ be a rational map. For all $y\in \Chat$, we have
\[
2+\ord_{y} (f_*\q)  \geq \min_{x\in f^{-1}(y)} \frac{2+\ord_{x} \q}{\deg_x f}.
\]
\end{lemma}

\begin{proof}
Let $V\subset \Chat$ be a topological disk containing $y$, such that $\vf\cap V\subseteq\{y\}$. For each $x\in f^{-1}(y)$, let $U_x$ be the connected component of $f^{-1}(V)$ containing $x$ and denote by $f_x:U_x\to V$ the restriction of $f$ to $U_x$. Note that $U_x$ is a topological disk since $f_x:U_x\to V$ is ramified at most above $y$. Then, 
\[f_*\q = \sum_{x\in f^{-1}(y)} \q_x \quad\text{with}\quad \q_x:= (f_x)_* \q.\]
In particular,  
\[\ord_y(f_*\q) \geq \min_{x\in f^{-1}(y)} \ord_y \q_x.\]
Let $\Gamma_x$ be the group of automorphisms of $U_x$ lifting the identity via $f_x:U_x\to V$. Those automorphisms fix $x$ and $\Gamma_x$ is a cyclic group of order $\deg_x f$. 
Observe that  
\[f_x^*\q_x = \sum_{\gamma \in \Gamma_x} \gamma^* \q.\]
In addition, since all $\gamma\in \Gamma_x$ are automorphisms fixing $x$, we have
\[\ord_x (\gamma^*\q )= \ord_x \q.\]
According to Lemma \ref{ordlemmapull},
\[\deg_x f\cdot \bigl(2+\ord_y(\q_x)\bigr) = 2+\ord_x (f_x^* \q_x) \geq 2+\min_{\gamma\in \Gamma_x}\ord_x(\gamma^* \q) =2+ \ord_x \q.\]
We therefore have
\[2+\ord_y(f_*\q) \geq  \min_{x\in f^{-1}(y)} (2+ \ord_y \q_x)\geq \min_{x\in f^{-1}(y)} \frac{2+ \ord_x \q}{\deg_x f}.\qedhere\]
\end{proof}

\begin{corollary}\label{coro1}
Let $f:\Chat\to \Chat$ be a rational map, let $\q$ be a meromorphic quadratic differential on $\Chat$, and let $X$ be the set of poles of $\q$. The set of poles of $f_*\q$ is contained in $f(X)\cup \vf$. 
\end{corollary}

\begin{proof}
If $f_*\q$ has a pole at $y\in \Chat$, then there is an $x\in f^{-1}(y)$ such that 
\[\frac{2+\ord_x \q}{\deg_x f} \leq 2+\ord_y(f_*\q)\leq 1.\]
Either $\deg_x f = 1$ and in that case, $\ord_x \q\leq 1-2 = -1$, so that $x\in X$ and $y\in f(X)$. 
Or $\deg_x f\geq 2$ and $y\in \vf$. 
\end{proof}

\begin{corollary}\label{coro2}
Let $f:\Chat\to \Chat$ be a rational map. If $\q\in \cQ(\Chat)$, then  $f_*\q\in \cQ(\Chat)$.
\end{corollary}

\begin{proof}
If $\q\in \cQ(\Chat)$, then for all $y\in \Chat$, we have 
\[\ord_y (f_*\q) \geq -2+ \min_{x\in f^{-1}(y)} \frac{2+\ord_{x} \q}{\deg_x f}>-2\quad\text{since}\quad 2+\ord_x \q\geq 2-1>0.\]
Since $\ord_y (f_*\q)$ is an integer, it is at least equal to $-1$. 
Therefore, $f_*\q$ has at worst a simple pole at $y$. 
\end{proof}

\begin{corollary}\label{coro3}
Let $f:\Chat\to \Chat$ be a rational map fixing $\infty$. If  $\q\in \cQ_k(\C)$ for some $k\geq 0$, then $f_*\q\in \cQ_k(\C)$. 
\end{corollary}

\begin{proof}
As in Lemma \ref{ordlemmapush}, if $y\in \C$, then $f_*\q$ has at worst a simple pole at $y$, and we have 
\[\ord_\infty (f_*\q)\geq -2 + \frac{2+ \ord_\infty \q}{\deg_\infty f} \geq -2+\frac{2-k}{1}\geq -k.\qedhere\]
\end{proof}

In particular, if $f:\Chat\to \Chat$ is a rational map, then $f_*$ is an endomorphism of $\cQ(\Chat)$, and when $f$ fixes $\infty$ (for example when $f$ is a polynomial), $f_*$ restricts to an endomorphism of $\cQ_k(\C)$ for all $k\geq 0$. 

\subsection{Transposition}

If $\q$ is a quadratic differential on $U$ and $\btheta$ is a vector field on $U$, we may consider the $1$-form $\q\otimes \btheta$ defined on $U$ by its action on vector fields $\bxi$:
\[\q\otimes\btheta(\bxi) = \frac{1}{4}\bigl(\q(\btheta+\bxi)-\q(\btheta-\bxi)\bigr).\]
If 
\[\q = q\ (\dz)^2\quad \text{and}\quad \btheta=\theta\ \frac{\d}{\dz},\quad\text{then}\quad \q\otimes \btheta = q \theta \ \dz.\]

We shall use the following lemma which, in some sense, asserts that the transpose of pushing forward a quadratic differential is pulling back a vector field. 

\begin{lemma}\label{lem:pushpull}
Let $f:\Chat\to\Chat$ be a rational map, let $\btheta$ be a meromorphic vector field on $\Chat$, and let $\q$ be a meromorphic quadratic differential on $\Chat$. Then 
\[
\res\bigl((f_*\q)\otimes \btheta,y\bigr)\;\; =\sum_{x\in f^{-1}(y)}\res\bigl(\q\otimes f^*\btheta,x\bigr).
\]
\end{lemma}
\begin{proof}
Let $\gamma$ be a small loop around $y$ with basepoint $a$. Then 
\[
\int_\gamma (f_* \q) \otimes \btheta=\sum_g \int_{\gamma\sm\{a\}} (g^*\q)\otimes \btheta=\sum_g\int_{g(\gamma\sm\{a\})} \q \otimes f^*\btheta=\int_{f^{-1}(\gamma)} \q\otimes f^*\btheta,
\]
where the sum ranges over the inverse branches $g$ of $f$ defined on $\gamma\sm\{a\}$. 
\end{proof}

\section{The contraction principle}\label{sec:contraction}

\noindent We shall now present a transcendental argument based on the fact that quadratic differentials in $\cQ(\Chat)$ are integrable. 

If $\q$ is a quadratic differential on $U$, we denote by $|\q|$ the positive $(1,1)$-form on $U$ defined by 
\[|\q|(\btheta_1,\btheta_2) := \frac{1}{2} \bigl|\q(\btheta_1-{\rm i}\btheta_2)\bigr| - \frac{1}{2} \bigl|\q(\btheta_1+{\rm i}\btheta_2)\bigr|.\]
If $\q = q\ (\d\zeta)^2$, then 
\[|\q| = |q|\cdot  \frac{\rm i}{2} \d\zeta\wedge \d\bar\zeta.\]
We shall say that $\q$ is {\em integrable} on $U$ if 
\[\|\q\|_{L^1(U)}:=\int_U |\q|<\infty.\]
Note that $\q$ is integrable in a neighborhood of a pole if and only if the pole is simple.

The following results due to Thurston will be crucial for our purposes. 

\begin{lemma}[Contraction Principle]
Let $f:U\to V$ be a covering map and let $\q$ be an integrable quadratic differential on $U$. Then, 
\[\|f_*\q\|_{L^1(V)}\leq \|\q\|_{L^1(U)}\]
and equality holds if and only if $f^*(f_*\q)=\phi \ \q$
with $\phi:U\to [0,+\infty)$
a real and positive function. 
\end{lemma}

\begin{proof}
The proof is an immediate application of the triangle inequality: for any topological disk $V'\subset V$, we have 
\[\int_{V'}|f_*\q| = \int_{V'}\left|\sum g^*\q\right|
\leq \int_{V'}\sum  |g^* \q| = \sum \int_{V'} |g^* \q| = \int_{f^{-1}(V')} |\q|,
\]
where the sum ranges over the inverse branches of $f$ defined on $V'$. 
It follows that 
\[\int_V |f_*\q|\leq \int_{f^{-1}(V)}|\q| =  \int_{U}|\q| \]
with equality if and only if for all inverse branches $g$ of $f$, we have $g^*\q=\psi_g\ f_*\q$ for some positive function $\psi_g$. The result follows by setting $\phi(g(y)):=\psi_g(y)$. 
\end{proof}

\begin{corollary}\label{coro:lambdainD}
If $f\in \Rat_D$ is postcritically finite, then $f_*:\cQ(\Chat)\to \cQ(\Chat)$ is contracting. In particular, the eigenvalues of $f_*$ have modulus at most $1$. 
\end{corollary}

\begin{corollary}\label{coro:lambdainDforpoly}
Let $f:\C\to \C$ be a polynomial of degree $D$. Then for all $k\geq 0$, the eigenvalues of 
$f_*:\cQ_k(\C)\to \cQ_k(\C)$ have modulus less than $1$. 
\end{corollary}

\begin{proof}
Suppose $\lambda$ is an eigenvalue and $\q\in \cQ_k(\C)$ is an associated eigenvector; that is, $\q\neq 0$ and $f_*\q = \lambda \q$. 
Let $V$ be a sufficiently large disk so that $U:=f^{-1}(V)$ is compactly contained in $V$. Set $V':=V\sm \vf$ and $U':=f^{-1}(V')$. Then 
\[|\lambda| \cdot \| \q\|_{L^1(V)}=\|\lambda \q\|_{L^1(V')} = \|f_* \q\|_{L^1(V')}\leq \|\q\|_{L^1(U')} =\|\q\|_{L^1(U)}< \|\q\|_{L^1(V)}.\]
The first inequality is an application of the contraction principle. 
The last inequality is strict since $U$ is compactly contained in $V$ and $\q\neq 0$. In addition, $ \| \q\|_{L^1(V)}>0$, so $|\lambda|<1$. 
\end{proof}

\begin{proposition}[Thurston]\label{prop:noeigen1}
Let $f\in \mathrm{Rat}_d$, and suppose $\lambda$ is  an eigenvalue of $f_*:\cQ(\Chat)\to \cQ(\Chat)$. If $|\lambda|=1$, then $f$ is a Lattès map with four postcritical points. If $\lambda=1$, then $f$ is a flexible Lattès map. 
\end{proposition}

\begin{proof}
Suppose $\q\in\cQ(\Chat)$ is an eigenvector associated to $\lambda$. Set $V:=\Chat\setminus \vf$ and $U:=f^{-1}(V)$, so that $f:U\to V$ is a covering map. Then,  
\[\|\q\|_{L^1(\Chat)}= \|\lambda \q\|_{L^1(V)} = \|f_*\q\|_{L^1(V)} \leq\|\q\|_{L^1(U)}=\|\q\|_{L^1(\Chat)}.\]
We therefore have the case of equality in the contraction principle, which shows that 
\[\lambda f^*\q =  f^*(f_*\q) = \phi\ \q\quad\text{with}\quad \phi:U\to [0,+\infty)\]
a real and positive function. Since $\q$ and $f^*\q$ are meromorphic, the function $\phi$ is meromorphic too, therefore constant. Since $\|f^*\q\|_{L^1(\Chat)} = D\ \|\q\|_{L^1(\Chat)}$, this constant is $\phi = 1/D$, and we have 
\[\q= D \lambda \ f^*\q.\]

As a consequence, the preimage of a pole of $\q$ by $f$ is either a pole of $\q$ or a critical point of $f$. In other words, if we denote by $Y$ the set of poles of $\q$, we have 
\[f^{-1}(Y)\subseteq Y\cup \cf.\]
Let $N$ be the number of critical points of $f$ contained in $f^{-1}(Y)$, counting multiplicities. On the one hand, 
\[D\cdot \card(Y) - N = \card\bigl(f^{-1}(Y)\bigr) \leq \card(Y) + \card(\cf)  \leq \card(Y)+2D-2.\]
On the other hand, $\q$ has at least four poles, so that $\card(Y)\geq 4$. 
Thus, 
\[ 4D-4 \leq (D-1)\cdot \card(Y)\leq N+\card(\cf)\leq 4D-4.\]
This shows that $\card(Y)=4$, $\card(\cf) = 2D-2$, so that all the critical points are simple, and $f^{-1}(Y) = Y\sqcup \cf$. 
This is a characterization of Lattès maps with four postcritical points (see Proposition \ref{prop:caraclattes}). 

There is a complex torus $\bE$, a ramified cover $\Theta :\bE\to \Chat$ ramifying at each point above $Y$ with local degree $2$, and an endomorphism $L:\bE\to \bE$ such that the following diagram commutes: 
\[\diagram
\bE\rto^L\dto_\Theta & \bE \dto^\Theta\\
\Chat\rto_f & \Chat
\enddiagram
\] 
Since $\Theta$ has local degree $2$ at each point above $Y$, the quadratic differential $\Theta^*\q$ has no pole, so it is globally holomorphic. 
In addition, 
\[L^*(\Theta^*\q) = \Theta^*(f^*\q) = \Theta^*\left(\frac{1}{D\lambda} \cdot \q\right)=\frac{1}{D\lambda} \Theta^*\q.\]
Thus, $L:\bE\to \bE$ is multiplication by $±\sqrt{D\lambda}$. In particular, if $\lambda=1$, then $f$ is a flexible Lattès map. 
\end{proof}
 
 \section{Infinitesimal rigidity}\label{sec:rigidity}
 
\noindent For completeness, we recall the proof by Thurston that with the exception of flexible Lattès maps, postcritically finite rational maps are infinitesimally rigid. Here and henceforth, we consider holomorphic families $t\mapsto \gamma_t$ defined near $t=0$ in $\C$. We shall employ the notation 
\[\gamma:=\gamma_0\and \dot\gamma:=\frac{\d \gamma_t}{\d t}\Big|_{t=0}.\]

If $t\mapsto f_t$ is a family of rational maps of degree $D$, then $\bxi:=\dot f\in \T_f\Rat_D$ is a section of the pullback bundle $f^\star \T\Chat$: for each $z\in \Chat$, $\bxi(z)\in T_{f(z)}\Chat$. Setting 
\[\btau(x):= (\D_x f)^{-1}\bigl( \bxi(x)\bigr)\quad\text{if}\quad x\not\in \cf,\]
we define a meromorphic vector field $\btau$ on $\Chat$, holomorphic outside $\cf$, with poles of order at most the multiplicity of $x$ as a critical point of $f$ when $x\in \cf$. This vector field satisfies
\[\bxi = \D f\circ \btau.\]

 \begin{theorem}[Thurston]\label{th:infrigidity}
 Let $t\mapsto f_t$ be a holomorphic family of postcritically finite rational maps of degree $D$, parameterized by a neighborhood of $0$ in $\C$. Then either the maps are flexible Lattès maps, or there is an analytic family of Möbius transformations $t\mapsto M_t$ such that $M_0={\rm id}$ and $f_t\circ M_t = M_t\circ f_0$. 
 \end{theorem}
 
\begin{proof}
Without loss of generality, we may assume that $f$ is not a flexible Lattès map. The fixed points of $f_t$ are superattracting or repelling and depend holomorphically on $t$. There are $D+1\geq 3$ such fixed points. Conjugating the family $t\mapsto f_t$ with a holomorphic family $t\mapsto M_t$ of Möbius transformations, we may assume that $f_t$ fixes $0$, $1$ and $\infty$. We will show that in this case, the holomorphic family $t\mapsto f_t$ is constant. 
It is enough to show that $\ds\frac{\d f_t}{\d t}$ identically vanishes, and since $t=0$ plays no particular role, it is enough to prove that $\dot f\equiv 0$. 

Set $\bxi:=\dot f\in \T_f\Rat_D$ and let $\btau$ be the globally meromorphic vector field on $\Chat$ such that $\bxi= \D f\circ \btau$.  

As $t$ varies, the set $Y_t:={\cal P}_{f_t}\cup \{0,1,\infty\}$ moves holomorphically and $f_t(Y_t) = Y_t$. For each $y\in Y$, let $t\mapsto y_t$ be the holomorphic curve satisfying $y_0=y$ and $y_t\in Y_t$. Set 
\[\bvtheta(y):=\frac{\d y_t}{\d t}\Big|_{t=0} \in \T_y\Chat.\] 
If $y\in Y$ and $z:=f(y)\in Y$, then $z_t= f_t(y_t)$, so that 
\[\bvtheta\circ f = \bxi + \D f\circ \bvtheta\quad \text{on } Y\quad \text{and}\quad \bvtheta\circ f = \bxi \quad \text{on }\cf.\]
Let $\btheta$ be a vector field, defined and holomorphic near $Y$, with $\btheta|_{Y}= \bvtheta$. Then, 
$f^*\btheta- \btau$ is holomorphic near $f^{-1}(Y)$ and coincides with $\btheta$ on $Y$. Also note that since $f_t$ fixes $0$, $1$ and $\infty$, $\btheta$ vanishes at $0$, $1$ and $\infty$. 

Let $\nabla_f:=\id-f_*:\cQ(\Chat;Y)\to \cQ(\Chat;Y)$. Observe that for $\q\in \cQ(\Chat;Y)$, 
\begin{align*}
\sum_{y\in Y} \res\bigl((\nabla_f \q)\otimes \btheta,y\bigr)  
&  \underset{(1)}= \sum_{y\in Y} \res(\q\otimes \btheta,y) -  \sum_{y\in Y} \res\bigl((f_*\q)\otimes \btheta,y\bigr)\\
& \underset{(2)}= \sum_{y\in Y} \res(\q\otimes \btheta,y) -  \sum_{x\in f^{-1}(Y)} \res(\q\otimes f^*\btheta,x)\\
&  \underset{(3)}= \sum_{y\in Y} \res(\q\otimes \btheta,y) -  \sum_{x\in f^{-1}(Y)} \res(\q\otimes (f^*\btheta-\btau),x)\\
&  \underset{(4)}= 0.
\end{align*}
Equality (1) holds by definition of $\nabla_f$; Equality (2) follows from Lemma \ref{lem:pushpull}; Equality (3) follows from the fact that $\q\otimes \btau$ is globally meromorphic on $\Chat$ with poles contained in $Y\cup \cf\subseteq f^{-1}(Y)$, so that the sum of its residues on $f^{-1}(Y)$ is $0$; Equality (4) follows from the fact that $f^*\btheta- \btau$ is holomorphic near $f^{-1}(Y)$ and coincides with $\btheta$ on $Y$ which contains the set of poles of $\q$. 

According to Proposition \ref{prop:noeigen1}, since $f$ is not a flexible Lattès map, $\lambda=1$ is not an eigenvalue of $f_*:\cQ(\Chat)\to \cQ(\Chat)$.  The operator $\nabla_f:\cQ(\Chat;Y)\to \cQ(\Chat;Y)$ is therefore injective, thus surjective. It follows that for any $\q\in \cQ(\Chat;Y)$, 
\[\sum_{y\in Y} \res(\q\otimes \btheta,y) = 0.\]
Equivalently, $\bvtheta$ is the restriction to $Y$ of a globally holomorphic vector field $\btheta$. Since $\btheta$ vanishes at $0$, $1$ and $\infty$, we have  $\btheta= 0$. The vector field $-\btau = f^*\btheta-\btau$ is globally holomorphic and coincides with $\btheta=0$ on $Y$. So $\btau=0$  and $\bxi=\D f\circ \btheta=0$ as required. 
 \end{proof}

The following corollary of infinitesimal rigidity is well known. We include a proof for completeness.

\begin{proposition}\label{prop:algrep}
If $f\in \Rat_D$ is postcritically finite but not a flexible Lattès map, then the Möbius conjugacy class of $f$ contains a representative with algebraic coefficients.
\end{proposition}
 
\begin{proof}
As in the previous proof, conjugating $f$ with a Möbius transformation if necessary, we may assume that $f$ fixes $0$, $1$ and $\infty$ and set $Y:=\pf\cup \{0,1,\infty\}$. In addition, set $X:=f^{-1}(Y)$ and let $\delta:X\to \N$ be defined by
\[\delta(x) :=\deg_xf.\]

Let us identify $\Chat$ and $\P^1(\C)$ via the usual map $\P^1(\C)\ni [u:v]\mapsto u/v\in \Chat$. For $N\geq 1$, let us denote by $\H_N$ the vector space of homogeneous polynomials of degree $N$ from $\C^2$ to $\C$. There is a canonical isomorphism between $\P^1(\C)$ and $\P(\H_1)$: a point of $\P^1(\C)$, a $1$-dimensional linear subspace of $\C^2$, is identified with the space of forms on $\C^2$ vanishing on this linear subspace; that is, a $1$-dimensional linear subspace of $\H_1$. This subsequently yields an identification of $\Chat$ with $\P(\H_1)$.

Note that $\Rat_D$ may be identified with the open subset of $\P(\H_D\times \H_D)$ corresponding to pairs of coprime homogeneous polynomials of degree $D$. Such a pair of polynomials defines a nondegenerate homogeneous polynomial map $\C^2\to\C^2$ of degree $D$ which induces an endomorphism $\P^1(\C)\to \P^1(\C)$ of degree $D$. 

We shall denote as $\Chat_D$ the $D$-fold symmetric product of the Riemann sphere; that is, the quotient of $\Chat^D$ by the group of permutation of the coordinates. The map 
\[\H_1^D\ni (P_1,\ldots,P_D)\mapsto P_1\times \cdots\times P_D\in \H_D\]
induces an identification between $\Chat_D$ and $\P(\H_D)$. 

Set 
\[
{\cal X} := \Rat_D\times \Chat^X\times \Chat^Y\quad \text{and}\quad {\cal Y}:= ( \Chat_D\times \Chat)^Y.\]
A point $(g,\alpha,\beta)\in {\cal X}$ may be represented by a triple $\bigl(G,(A_x)_{x\in X},(B_y)_{y\in Y}\bigr)$ where
$G:=(G_1,G_2)\in \H_D\times \H_D$ is a nondegenerate homogeneous polynomial map $\C^2\to \C^2$ of degree $D$, and where $A_x\in \H_1$ and $B_y\in \H_1$ are linear forms $\C^2\to \C$. 
Recall that $Y\subset X$ and consider the algebraic map $\Phi:{\cal X}\to {\cal Y}$ induced by 
\[\bigl(G,(A_x)_{x\in X},(B_y)_{y\in Y}\bigr)\mapsto \left(\prod_{x\in f^{-1}(y)} A_x^{\delta(x)},A_y\right)_{y\in Y}\]
and the algebraic map $\Psi :{\cal X}\to {\cal Y}$ induced by 
\[\bigl(G,(A_x)_{x\in X},(B_y)_{y\in Y}\bigr)\mapsto \left(B_y\circ G,B_y\right)_{y\in Y}.\]
Let us consider the algebraic set ${\cal Z}\subset {\cal X}$ defined by the equation $\Phi=\Psi$.  

We claim that the triple $(g,\alpha,\beta)$ belongs to ${\cal Z}$ if and only if we have a commutative diagram 
\[\diagram
X\rto^\alpha \dto_f & \Chat\dto^g\\
Y\rto_\beta &\Chat
\enddiagram
\quad\text{with}\quad  \deg_{\alpha(x)}g = \deg_xf,\quad \text{and}\quad \alpha|_Y = \beta|_Y.\]
Indeed,
\[B_y\circ G = \prod_{x\in f^{-1}(y)} A_x^{\delta(x)} \]
if and only if for each $x\in f^{-1}(y)$, the point $\alpha(x)\in \Chat$ is a preimage of $\beta(y)\in \Chat$ by $g$ taken with multiplicity $\delta(x) = \deg_x f$. In addition, $A_y = B_y$ if and only if $\alpha(y) = \beta(y)$. 

As a consequence,  if $(g,\alpha,\beta)\in {\cal Z}$, then ${\cal C}_g = \alpha({\cal C}_f)\subseteq \alpha(X)$. In that case,  ${\cal P}_g = \beta (\pf)\subseteq \beta(Y)$, and $g$ is postcritically finite.

Set 
\[{\cal Z}_0:=\bigl\{(g,\alpha,\beta)\in {\cal Z}~|~\alpha(0)=0,~\alpha(1)=1,\text{ and }\alpha(\infty)=\infty\bigr\}.\]
Note that the triple $(f,\id,\id)$ belongs to ${\cal Z}_0$. 
According to Theorem \ref{th:infrigidity}, if $f$ is not a flexible Lattès map, then the algebraic set ${\cal Z}_0$ has dimension $0$ at $f$. This implies that $f$ has algebraic coefficients. 
\end{proof}
 
\section{The eigenvalues of $f_*:\cQ(\Chat)\to \cQ(\Chat)$}\label{sec:qchat}

\begin{proposition}\label{0inLambda}
If $f\in \Rat_D$, then $0$ is an eigenvalue of $f_*:\cQ(\Chat)\to \cQ(\Chat)$.
\end{proposition}

\begin{proof}
Let $Y$ be a subset of $\Chat\sm \vf$ with $\card(Y)=3$. Set $X:=f^{-1}(Y)$. Note that $X$ is disjoint from $\cf$ and 
\[\card(X)-\card(\cf) \geq  3D - (2D-2)\geq D+2\geq 4.\]
So there is a nonzero quadratic differential $\q$ which vanishes on $\cf$ and whose poles are simple and contained in $X$. 
According to Lemma \ref{ordlemmapush}, the quadratic differential $f_*\q$ is holomorphic on $\Chat\sm Y$ and has at most simple poles along $Y$. Thus, it has at most $3$ poles counting multiplicities, which forces $f_*\q=0$. 
\end{proof}

Let us now assume $f\in \Rat_D$ is postcritically finite and set 
\[\qf:=\cQ(\Chat;\pf).\]
Note that $f_*(\qf)\subseteq \qf$. Indeed, according to Corollary \ref{coro1}, if $\q\in \qf$, then the poles of $f_*\q$ are contained in $f(\pf) \cup \vf= \pf$. So, the set of eigenvalues of $f_*:\cQ(\Chat)\to \cQ(\Chat)$ may be written as the union $\spec_f\cup \Lambda_f$ with 
\[\spec_f:=\{\lambda\in \C~|~\lambda\text{ is an eigenvalue of }f_*:\qf\to \qf\},\]
and 
\[\Lambda_f:=\{\lambda\in \C~|~\lambda\text{ is an eigenvalue of }f_*:\cQ(\Chat)/\qf\to\cQ(\Chat)/ \qf\}.\]
We postpone the study of $\spec_f$ and focus now on $\Lambda_f$. 

\begin{proposition}\label{prop:eigmult}
The elements of $\Lambda_f\sm\{0\}$ are the complex numbers $\lambda$ such that $1/\lambda^m$ is the multiplier of a cycle of $f$ of period $m$ which is not contained in $\pf$. 
\end{proposition}

\begin{proof}
To begin with, let us describe the space $\cQ(\Chat)/\qf$. First, observe that if $\q\in \cQ(\Chat)$, then the residue  of $\q$ at a point $x\in \Chat$ is naturally a form on $\T_x\Chat$; that is, an element of $\T^*_x\Chat$. It may be defined as follows: if $\theta\in \T_x\Chat$ and $\btheta$ is a vector field defined and holomorphic near $x$ with $\btheta(x)  = \theta$, then 
\[\res(\q,x)(\theta):= \res(\q\otimes \btheta,x).\]
The result does not depend on the extension $\btheta$ since if $\btheta_1$ and $\btheta_2$ are two holomorphic vector fields which coincide at $x$, then $\btheta_1-\btheta_2$ vanishes at $x$, so that $\q\otimes (\btheta_1-\btheta_2)$ is holomorphic near $x$ and
$\res(\q\otimes \btheta_1,x) = \res(\q\otimes \btheta_2,x)$. 

Second, set 
\[B:=\begin{cases}
\pf&\text{ if }\card(\pf)\geq 3,\\
\pf\cup \{\alpha\}\text{ with }\alpha\text{ a repelling fixed point of }f&\text{ if } \card(\pf)=2.
\end{cases}\]
So $\card(B)\geq 3$, $f(B)=B$, and $\qf = \cQ(\Chat;B)$. Set 
\[\Omega_f:=\bigoplus_{x\in \Chat\sm B} \T^*_x\Chat.\]
Note that $\Omega_f$ is the space of $1$-forms on $\Chat\sm B$ which vanish outside a finite set. 
Consider the map $\Res:\cQ(\Chat)/\qf \to \Omega_f$ defined by 
\[\Res\bigl([\q]\bigr)(x):= \res(\q,x).\]
This map is well defined since if $\q_1\in \cQ(\Chat)$ and $\q_2\in \cQ(\Chat)$ satisfy $\q_1-\q_2\in \qf$, then  $\q_1-\q_2$ is holomorphic on $\Chat\sm \pf$, so that $\res(\q_1,x) = \res(\q_2,x)$ for all $x\in \Chat\sm B$.

\begin{lemma}
The map $\Res :\cQ(\Chat)/\qf \to \Omega_f$ is an isomorphism. 
\end{lemma}

\begin{proof}
First, if $\q\in \cQ(\Chat)$ and  $\res(\q,x) = 0$ for all $x\in \Chat\sm B$, then $\q$ is holomorphic outside $B$, so that $\q\in \cQ(\Chat;B) = \qf$. It follows that $\Res:\cQ(\Chat)/\qf \to \Omega_f$ is injective. 

Second, given two distinct points $x_1$ and $x_2$ in $\Chat$, let $\bomega_{x_1,x_2}$ be the meromorphic $1$-form on $\Chat$ which is holomorphic outside $\{x_1,x_2\}$, has residue $1$ at $x_1$ and residue $-1$ at $x_2$:
\[\bomega_{x_1,x_2} = \begin{cases} \dz/(z-x_1) - \dz/(z-x_2) & \text{ if } x_1\neq \infty\text{ and } x_2\neq \infty,\\
\dz/(z-x_1)&\text{ if }x_2=\infty,\\
-\dz/(z-x_2)&\text{ if }x_1=\infty.
\end{cases}
\]
Note that $\bomega_{x_1,x_2}$ does not vanish. 

Third, choose three distinct points $x_1$, $x_2$ and $x_3$ in $B$. Given $\omega\in \Omega_f$, we may define a function $\phi:\Chat\sm B\to \C$ by 
\[\omega(x) = \phi(x)\cdot \bomega_{x_2,x_3}(x).\]
Since $\omega$ vanishes outside a finite set, $\phi$ also vanishes outside a finite set.
Set 
\[\q:=\sum_{x\in \Chat\sm B}\phi(x)\cdot \bomega_{x,x_1}\otimes \bomega_{x_2,x_3}.\]
Since $x$, $x_1$, $x_2$, and $x_3$ are distinct, $\bomega_{x,x_1}\otimes \bomega_{x_2,x_3}\in \cQ(\Chat)$. Since  $\phi$ vanishes outside a finite set, the sum is finite and $\q\in \cQ(\Chat)$. By construction, for $x\in \Chat\sm B$, 
\[\res(\q,x) = \phi(x)\cdot \bomega_{x_2,x_3}(x) = \omega(x),\]
so that $\Res\bigl([\q]\bigr) = \omega$. 
It follows that $\Res:\cQ(\Chat)/\qf \to \Omega_f$ is surjective. 
\end{proof}

Observe that if $y\in \Chat\sm B\subseteq \Chat\sm \vf$, then ${\rm D}_xf:\T_x\Chat\to \T_y\Chat$ is invertible for any $x\in f^{-1}(y)$, and $f_*:\cQ(\Chat)/\qf\to\cQ(\Chat)/ \qf$ is conjugate to the linear map 
$f_*:\Omega_f\to \Omega_f$ defined by 
\[f_*\omega(y) :=\sum_{x\in f^{-1}(y)} \omega \circ ({\rm D}_xf)^{-1}.\]

Next, let $X\subset \Chat\sm B$ be a cycle of $f$ of period $m$ and multiplier $\mu$. Note that the space 
\[E_X:=\bigoplus_{x\in X} \T^*_x\Chat\subset \Omega_f\]
has dimension $m$ and is invariant by $f_*:\Omega_f\to \Omega_f$.

\begin{lemma}\label{lem:eigenvalscycle}
The endomorphism $f_*:E_X\to E_X$ is diagonalizable. Its eigenvalues are the $m$-th roots of $1/\mu$. 
\end{lemma}

\begin{proof}
Suppose $\lambda^m=1/\mu$, and let $x_0\mapsto x_1\mapsto \ldots\mapsto x_{m-1}\mapsto x_m=x_0$ be the points of $X$. Let $\omega_0\in\T^*_{x_0}\Chat$ be any nonzero form. For $1\leq j\leq m$, define recursively 
\[\omega_j:=\lambda \omega_{j-1}\circ ({\rm D}_{x_{j-1}}f)^{-1}\in \T^*_{x_j}\Chat.\] 
Then 
\[\omega_{m} = \lambda^m\omega_0\circ ({\rm D}_{x_0}f^{\circ m})^{-1} = \frac{\lambda^m}{\mu}\omega_0 = \omega_0\]
since ${\rm D}_{x_0}f^{\circ m}:\T_{x_0}\Chat\to \T_{x_0}\Chat$ is multiplication by $\mu$ and $\lambda^m=1/\mu$. It follows that the $1$-form $\omega\in \Omega_f$ defined by 
\[\omega(x) = \begin{cases}0&\text{ if } x\not\in X,\\
\omega_j&\text{ if } x = x_j\in X
\end{cases}\]
satisfies $f_*\omega = \lambda\omega$. 
\end{proof}

Finally,  assume $\lambda\neq 0$ is an eigenvalue of $f_*:\Omega_f\to \Omega_f$  and let $\omega\in \Omega_f$ be an eigenvector associated to $\lambda$. Set $X:=\{x\in \Chat~|~\omega(x)\neq 0\}$. If $\omega(y)\neq 0$, then there exists $x\in f^{-1}(y)$ such that $\omega(x)\neq 0$.  Thus, $X\subseteq f(X)$ and since the cardinality of $f(X)$ is always less than or equal to the cardinality of $X$, we necessarily have $X = f(X)$. So $X\subset \Chat\sm B$ is a union of cycles of $f$. It follows from Lemma \ref{lem:eigenvalscycle} that the eigenvalues of $f_*:\cQ(\Chat)/\qf\to\cQ(\Chat)/ \qf$ other than $\lambda=0$ are the complex numbers $\lambda$ such that $1/\lambda^m$ is the multiplier of a cycle of $f$ of period $m$ which is not contained in $B$. This completes the proof of Proposition \ref{prop:eigmult} when $B=\pf$; that is, when $\card(\pf)\geq 3$. 

To complete the proof of Proposition \ref{prop:eigmult} when $\card(\pf)=2$, observe that
\begin{itemize}
\item either $f$ is conjugate to $z\mapsto 1/z^D$, in which case there are $D+1$ repelling fixed points, each with multiplier $\mu=t-D$, so that the multipliers of the fixed points which are not contained in $B$ are the multipliers of the fixed points which are not contained in $\pf$; 
\item or $f$ is conjugate to $z\mapsto z^D$, in which case there may be a single repelling fixed point at $z=1$ when $D=2$; its multiplier is $2$ and we must show that $1/2$ is an eigenvalue of $f_*:\cQ(\Chat)/\qf\to \cQ(\Chat)/\qf$; in that case the multiplier of the cycle of period $m=2$ is $\mu = 4$, so that $1/2$ is a $m$-th root of $1/\mu$. 
\end{itemize}
This completes the proof of Proposition \ref{prop:eigmult}. 
\end{proof}

\begin{corollary}\label{coro:lambda}
Let $f\in \Rat_D$ be postcritically finite. Then, $\Lambda_f\subset \Disk$. 
\end{corollary}

\begin{proof}
A cycle of $f$ not contained in $\pf$ is repelling. 
\end{proof}

\begin{example}
If $f(z) = z^{±D}$, then a cycle of period $m$ not contained in $\pf$ is  a repelling cycle of multiplier $(±D)^m$ and there is at least one such cycle for each period $m\geq 2$. It follows that 
\[
\Lambda_f = \{0\}\cup \left\{\frac{{\rm e}^{2\pi {\rm i}p/q}}{D}, ~p/q\in \Q/\Z\right\}.
\]
Note that $\card(\pf) = 2$, so that $\qf=\{0\}$ and $\spec_f = \emptyset$. 
\end{example}

\begin{example}
If $f\in \Rat_D$ is a flexible Lattès map, then a cycle of period $m$ not contained in $\pf$ is  a repelling cycle of multiplier $\sqrt{D}^m$ and there is at least one such cycle for each period $m\geq 1$. It follows that 
\[
\Lambda_f = \{0\}\cup \left\{\frac{{\rm e}^{2\pi {\rm i}p/q}}{\sqrt{D}}, ~p/q\in \Q/\Z\right\}.
\] 
\end{example}

\begin{proposition}\label{prop:algnum}
Let $f\in \Rat_D$ be postcritically finite. If $\lambda\in \Lambda_f\sm\{0\}$, then $\lambda$ is an algebraic number but not an algebraic integer. 
\end{proposition}

\begin{proof}
As discussed in the previous example, the proposition holds for flexible Lattès maps, so we may assume $f$ is not a flexible Lattès map in this proof. 

If $f\in \Rat_D$ and $g\in \Rat_D$ are conjugate by a Möbius transformation $M$; that is, $M\circ f = g\circ M$, then the linear map $M_*:\cQ_f\to \cQ_g$ conjugates $f_*:\qf\to \qf$ to $g_*:\cQ_g\to \cQ_g$: 
\[g_*(M_* \q) = M_*( f_*\q )= M_*(\lambda \q) = \lambda M_*\q.\]
 Thus, $\Lambda_f = \Lambda_g$. 

According to Proposition \ref{prop:algrep}, the conjugacy class of $f$ contains a representative with algebraic coefficients. Without loss of generality, we may therefore assume that this is the case for $f$ and consider $f$ as a rational map $f:\P^1(\overline \Q)\to \P^1(\overline \Q)$. Working over the algebraically closed field $\Qbar$, we deduce that the multipliers of cycles of $f$ are algebraic numbers, so $\Lambda_f$ consists of algebraic numbers. 

If $\sigma:\Qbar\to \Qbar$ is a Galois automorphism, then $g:=\sigma\circ f\circ \sigma^{-1}:\P^1(\Qbar)\to \P^1(\Qbar)$ is postcritically finite. In addition, $f_*\q = \lambda \q$ if and only 
\[g_*(\sigma_* \q) = \sigma_*( f_*\q )= \sigma_*(\lambda \q) = \sigma(\lambda)\  \sigma_*\q.\]
Thus, $\lambda\in \Lambda_f$ if and only if $\sigma(\lambda)\in \Lambda_g$. 
According to Corollary \ref{coro:lambda}, if $\lambda\in \Lambda_f$, then $\bigl|\sigma(\lambda)\bigr|<1$ for any Galois automorphism $\sigma:\Qbar\to \Qbar$. Thus, if $\lambda$ is an algebraic integer, the product of $\lambda$ with its Galois conjugates is an integer of modulus less than $1$, which forces $\lambda=0$.
\end{proof}

\section{The eigenvalues of $f_*:\qf\to \qf$}\label{sec:eigenqf}

\noindent From now on, we assume that $\card(\pf)\geq 4$ so that $\qf$ is not reduced to $\{0\}$. In that case, the dimension of $\qf$ is $\card(\pf)-3$, so $f_*:\qf\to \qf$ has at most $\card(\pf)-3$ eigenvalues. 

\subsection{Lattès maps with four postcritical points}

According to Proposition \ref{prop:noeigen1}, if $\spec_f$ contains an eigenvalue of modulus $1$, then $f$ is a Lattès map with $\card(\pf)=4$. The converse is also true. 

\begin{proposition}
Suppose $f\in \Rat_D$ is a Lattès map with $\card(\pf)=4$. Then $f_*:\qf\to \qf$ is multiplication by $\lambda$ with $|\lambda|=1$.  In addition, $\lambda =\pm 1$, or $\lambda$ is a quadratic number. Any imaginary quadratic number of modulus $1$ may arise for some Lattès map.  If $\lambda$ is an algebraic integer, then $\lambda$ is a root of unity of order 1, 2, 3, 4, or 6.
\end{proposition}

\begin{proof}
Since $\card(\pf)=4$, the dimension of $\qf$ is $1$, so $f_*:\qf\to \qf$ has a unique eigenvalue, and $f_*$ is multiplication by this eigenvalue. By assumption, there is a complex torus $\bE$, a ramified cover $\Theta :\bE\to \Chat$ ramifying at each point above $\pf$ with local degree $2$, and an endomorphism $L:\bE\to \bE$ such that the following diagram commutes: 
\[\diagram
\bE\rto^L\dto_\Theta & \bE \dto^\Theta\\
\Chat\rto_f & \Chat.
\enddiagram
\] 
As mentioned in Proposition \ref{prop:noeigen1}, if $\q\in \qf$, then $\Theta^*\q$ is a multiple of $(\dz)^2$ and if $\lambda$ is the eigenvalue of $f_*:\qf\to \qf$, then $\ds L_*(\Theta^*\q) = \frac{1}{D\lambda} \Theta^*\q$. Thus, if $L(w) = \alpha w + \beta$, then $ \alpha^2= D\lambda$. 

According to \cite[\S5]{jacklattes}, we have $|\alpha|^2=D$, so that $|\lambda|=1$. In addition, either
\begin{itemize}
\item $\alpha\in \Z$ in which case $\lambda=1$, and $f$ is a flexible Lattès map, or
\item $\alpha$ is an imaginary quadratic integer; that is, $\alpha^2-C\alpha+D =0$ with $C\in \Z$ and $C^2<4D$.  
\end{itemize}
In the latter case, $\lambda=\alpha^2/D\in \Q[\alpha]$ is either $-1$ or an imaginary quadratic number of modulus $1$. 

Conversely, suppose $\lambda=-1$ or $\lambda$ is an imaginary quadratic number of modulus $1$. Let $k\geq 2$ be a sufficiently large integer so that $\alpha:=k\sqrt{\lambda}$ is an imaginary quadratic integer and set $D:=k^2$. 
According to \cite[\S5]{jacklattes}, there exists a Lattès map $f\in \Rat_D$ with $L(w) = \alpha w$. According to the previous discussion, the eigenvalue of $f_*:\qf\to \qf$ is $\lambda$.

Finally, if $\lambda$ is a quadratic integer, then it is a unit since $|\lambda|=1$. Thus, it is a root of unity of order 1, 2, 3, 4, or 6.
\end{proof}

\subsection{Non Lattès maps}

We now assume that $f$ is not a Lattès map. In that case, according to  Corollary \ref{coro:lambdainD} and Proposition \ref{prop:noeigen1}, the eigenvalues of $f_*:\qf\to \qf$ are contained in the unit disk. 

\begin{proposition}
Let $f\in \Rat_D$ be postcritically finite with $\card(\pf)\geq 4$, and suppose that $f$ is not a Lattès map. If $\lambda\in \spec_f\sm \{0\}$, then $\lambda$ is an algebraic number but not an algebraic integer.  
\end{proposition}

\begin{proof}
We proceed as in the proof of Proposition \ref{prop:algnum}. Conjugating $f$ with a Möbius transformation if necessary, we may assume that $f$ is a rational map $f:\P^1(\Qbar)\to \P^1(\Qbar)$. Working over $\Qbar$, we deduce that the eigenvalues of $f_*:\qf\to \qf$ are algebraic numbers. 

Let $\lambda\in \spec_f$, and let $\sigma:\Qbar\to \Qbar$ be a Galois automorphism. Then, $\sigma(\lambda)\in \spec_g$ with 
$g:=\sigma \circ f\circ \sigma^{-1}:\P^1(\Qbar)\to \P^1(\Qbar)$, so $\bigl|\sigma(\lambda)\bigr|<1$. 
Thus, if $\lambda$ is an algebraic integer, then $\lambda=0$. 
\end{proof}

\subsection{An example where $\spec_f=\{0\}$}

Proposition \ref{0inLambda} establishes that $0\in \Lambda_f$. However, $0$ does not necessarily belong to $\spec_f$. For example, if $f:\C\to \C$ is a polynomial with periodic critical points, then $0\notin\spec_f$ (see \cite{bekp}). 

We now present an example of a postcritically finite rational map $f$ for which $0\in \spec_f$; this example appears in \cite{bekp}. 

\begin{proposition}
Let $f:\Chat\to \Chat$ be the rational map given by
$\ds
f:z\mapsto \frac{3z^2}{2z^3+1}.
$
Then $\spec_f=\{0\}$. 
\end{proposition}

\begin{proof}
The critical set of $f$ is 
$\cf=\{0,1,\omega,\bar{\omega}\}$, where 
\[\omega:=
-1/2+{\rm i}\sqrt3/2\and \bar{\omega}:= -1/2-{\rm i}\sqrt3/2\] are cube roots
of unity. The postcritical set of $f$ is $\pf=\{0,1,\omega,\bar{\omega}\}$, and $f$ has the following postcritical dynamics. 
\[
\xymatrix{0\ar@(ur,dr)^2}\qquad\xymatrix{1\ar@(ur,dr)^2}\qquad
\xymatrix{\omega\ar@/^1.1pc/[r]^2 & \bar{\omega}\ar@/^1.1pc/[l]^2}
\]
Since $\card(\pf)=4$, the space $\qf$ is $1$-dimensional, and there is a single eigenvalue $\lambda$. 
Consider 
\[\q := \frac{(\dz)^2}{z(z^3-1)}\in \pf, \quad\text{so that}\quad f_*\q = \lambda \q.\]
Let $g:\Chat\to \Chat$ be the rotation $z\mapsto \omega z$. Then, 
\[f\circ g (z) = f(\omega z) = \omega^2 f(z) = g^{\circ  2}\circ f(z).\]
Setting $u=g(z) = \omega z$, we have that
\[g_*\q=g_*\left(\frac{(\dz)^2}{z(z^3-1)}\right) = \frac{(\d u)^2/\omega^2}{u/\omega \cdot (u^3-1)} = \frac{\q}{\omega}.\]
As a consequence, 
\[f_*(g_*\q) = f_*\left(\frac{\q}{\omega}\right) = \frac{f_*\q}{\omega}\quad \text{and}\quad 
g^{\circ 2}_*(f_*\q) = \frac{f_*\q}{\omega^2}.\]
It follows that
\[\frac{f_*\q}{\omega} = \frac{f_*\q}{\omega^2}\]
and since $\omega\neq \omega^2$, we necessarily have $f_*\q = 0$. 
\end{proof}

\section{Characteristic polynomials}\label{sec:charac}

\noindent In this section, the map $f:\Chat\to\Chat$ is postcritically finite with postcritical set $\pf$ and $f(\infty)=\infty$.
Let $d_\infty$ be the local degree of $f$ at $\infty$, and let $\mu_\infty$ be the multiplier of $f$ at $\infty$. Note that $d_\infty\geq 2$ if and only if $\mu_\infty=0$. 

Our goal is to compute the characteristic polynomial $\chi_f$ of $f_*:\qf\to \qf$:
\[\chi_f(\lambda):=\det(\lambda\cdot \id - f_*).\]
Set $X:=\pf\sm \{\infty\}$ and consider the square matrix $A_f$ whose coefficients $a_{y,x}$, indexed by $X\times X$, are defined by: 
\[a_{y,x}:=\sum_{ {w\in  f^{-1}(y)\cap (\cf\cup\{x\})}} \res\left(\frac{\d z}{(z-x)f'(z)},w\right).\]

\begin{proposition}\label{prop:charac}
We have that  $\det(\lambda\cdot {\rm I}-A_f) = \xi_f(\lambda)\cdot \chi_f(\lambda)$ with 
\[\xi_f(\lambda):=\begin{cases}
(\lambda - \mu_\infty)(\lambda-1/d_\infty)&\text{if }\infty\in \pf,\\
(\lambda - \mu_\infty)(\lambda-1/d_\infty)(\lambda-1/\mu_\infty)&\text{if }\infty\not \in \pf.
\end{cases}
\]
\end{proposition}

The remainder of \S\ref{sec:charac} is devoted to the proof of this proposition. We first outline a sketch of the proof. 

\medskip
\noindent{\bf Step 1.} 
Instead of working in $\qf$, we introduce the following vector spaces of meromorphic quadratic differentials:
\begin{itemize}
\item $\qf^1$ for those with at worst simple poles at the points in $\pf\cup\{\infty\}$, 
\item $\qf^2$ for those with at worst simple poles at points in $\pf$, and at worst a double pole at $\infty$, and
\item $\qf^3$ for those with at worst simple poles at points in $\pf$, and at worst a triple pole at $\infty$.
\end{itemize}
We will show that each of these spaces is invariant under $f_*$. As subspaces,
\[
\qf\subseteq\qf^1\subset\qf^2\subset\qf^3.
\]
If $\infty\in\pf$, then $\qf=\qf^1$. Otherwise $\dim \qf^1/\qf=1$. In all cases 
\[\dim \qf^3/\qf^2=\dim\qf^2/\qf^1=1.\]

\medskip
\noindent{\bf Step 2.} 
We will show that the eigenvalues of the induced endomorphisms 
\[
\qf^3/\qf^2\to \qf^3/\qf^2,\quad \qf^2/\qf^1 \to \qf^2/\qf^1\and  \qf^1/\qf \to \qf^1/\qf
\]
are given in Table \ref{extraeigentable}. 
\begin{table}[h]\begin{center}
\begin{tabular}{|c |c |c|c|}
\cline{2-4}
\multicolumn{1}{c|}{} & $\qf^3/\qf^2$ & $\qf^2/\qf^1$ & $\qf^1/\qf$ \\
\hline
$\infty\notin\pf$, $d_\infty=1$ & $\mu_\infty$ &1 &1$/\mu_\infty$ \\
\hline
$\infty\in\pf$, $d_\infty=1$  & $\mu_\infty$ & $1$  & none \\
\hline
$\infty\in\pf$, $d_\infty\geq 2$  & $0$ & $1/d_\infty$  & none \\
\hline
\end{tabular}
\end{center}\caption{The eigenvalues of the quotient maps induced by $f_*$ where $\mu_\infty$ is the multiplier of $f$ at $\infty$, and $d_\infty$ is the local degree of $f$ at $\infty$.\label{extraeigentable}}
\end{table}

\medskip
\noindent{\bf Step 3.} We will then compute the eigenvalues of $f_*:\qf^3\to \qf^3$ as follows. The quadratic differentials 
\[
\left\{\q_x:= \frac{\d z^2}{z-x}\right\}_{x\in\pf\sm\{\infty\}}
\]
form a basis of $\qf^3$. According to Lemma \ref{lem:ayx} below, the matrix of $f_*:\qf^3\to \qf^3$ in the basis $\{\q_x\}_{x\in\pf\sm\{\infty\}}$ is the matrix $A_f$.

\medskip
\noindent{\bf Step 4.}
For $k\in \{1,2,3\}$, let $\chi_f^k$ be the characteristic polynomial of $f_*:\qf^k\to \qf^k$, and let $\xi_f^k$ be the characteristic polynomial of $f_*:\qf^k/\qf^{k-1}\to \qf^k/\qf^{k-1}$, with the convention that $\qf^0 := \qf$ and $\xi_f^1=1$ if $\qf^1 = \qf$; that is, if $\infty\not\in \pf$. 

Since
\[
\qf\subseteq\qf^1\subset\qf^2\subset\qf^3 
\]
are invariant by $f_*$, we have 
\[\chi_f^3 = \xi_f^3\cdot \chi_f^2 =   \xi_f^3\cdot\xi_f^2\cdot \chi_f^1 = \xi_f^3\cdot\xi_f^2\cdot \xi_f^1\cdot \chi_f.\]
According to Step 2, 
\[\xi_f^3(\lambda) = \lambda - \mu_\infty,\quad  \xi_f^2(\lambda) = \lambda - \frac{1}{d_\infty},\quad \text{and when }\infty\not\in \pf,\quad  \xi_f^1(\lambda) = \lambda - \frac{1}{\mu_\infty}.\]
Therefore, $\xi_f^3\cdot\xi_f^2\cdot \xi_f^1 = \xi_f$ and $\chi_f^3 = \xi_f\cdot \chi_f$. 
Proposition \ref{prop:charac} follows from Step 3: 
\[ \det(\lambda\cdot {\rm I}-A_f) = \chi_f^3(\lambda) =  \xi_f(\lambda)\cdot \chi_f(\lambda) .\]

We now proceed with the proof working step by step. 

\subsection{Invariant subspaces}\label{invariantsection}

\begin{lemma}
The vector spaces $\qf^1,\qf^2$, and $\qf^3$ are invariant by $f_*$. 
\end{lemma}

\begin{proof}
Suppose $\q\in \qf^k$ with $k\in \{1,2,3\}$. According to Corollary \ref{coro1}, the poles of $f_*\q$ are contained in $f\bigl(\pf\cup \{\infty\}\bigr)\cup \vf = \pf\cup \{\infty\}$. 

Assume $k =1$. Then $\qf^1\subset \cQ(\Chat)$ and by Corollary \ref{coro2}, $f_*\q\in \cQ(\Chat)$, so that the poles of $f_*\q$ are simple. Thus, $f_*\q\in \qf^1$. In other words, $f_*(\qf^1)\subseteq \qf^1$. 

Assume $k\in \{2,3\}$ and $y\in \pf$.  On the one hand, if $x\in f^{-1}(y)\setminus\{\infty\}$, then $2+\ord_x \q\geq 1$ and
\[\frac{2+\ord_x\q}{\deg_x f}-2\geq \frac{1}{\deg_xf}-2>-2.\]
In particular, if $y\neq \infty$, then according to Lemma \ref{ordlemmapush}, $f_*\q$ has at worst a simple pole at $y$. 
On the other hand, if $x=\infty$, then $2+\ord_\infty\q\geq 2-k$ and since $2-k\leq 0$, 
\[\frac{2+\ord_\infty\q}{\deg_\infty f}-2\geq 2-k-2=-k.\]
It follows that $f_* \q$ has at worst a pole of order $k$ at $\infty$.
\end{proof}

\subsection{Extra eigenvalues} \label{subsec:extraeig}
Here, we identify the eigenvalues arising from the induced operators $f_*:\qf^k/\qf^{k-1}\to\qf^k/\qf^{k-1}$ for $k\in \{1,2,3\}$ (see Table \ref{extraeigentable}). The case of $\qf^1/\qf$ is covered by \S\ref{sec:qchat}, more precisely by Lemma  \ref{lem:eigenvalscycle}: if $\infty$ is a fixed point with multiplier $\mu_\infty$ not contained in $\pf$ then $1/\mu_\infty$ is an eigenvalue of $f_*:\qf^1/\qf\to \qf^1/\qf$. We therefore only need to deal with $\qf^2/\qf^1$ and $\qf^3/\qf^2$. 

Fix two vector fields $\btheta_2$ and $\btheta_3$, where $\btheta_k$ is holomorphic near $\infty$ and vanishes to order $k-1$ at $\infty$.  Let $\alpha_k:\qf^k\to \C$ be the form defined by
\[
\alpha_k(\q):= \res(\q\otimes \btheta_k,\infty).
\]
This form is in the annihilator of $\qf^{k-1}$, and as such, $\alpha_k$ may be canonically identified with an element in the dual of the quotient $\qf^k/\qf^{k-1}$. Therefore, if $\lambda_k$ is the eigenvalue of the induced operator  $f_*:\qf^k/\qf^{k-1}\to\qf^k/\qf^{k-1}$, then 
\[
\alpha_k(f_*\q)=\lambda_k\ \alpha_k(\q).
\]
The data presented in Table \ref{extraeigentable} is a consequence of the following lemma. 
\begin{lemma} If $\q\in \qf^3$, then 
\[\alpha_3(f_*\q)=\mu_\infty \ \alpha_3(\q).\] 
If  $\q\in \qf^2$, then 
\[\alpha_2(f_*\q)=\frac{1}{d_\infty}\ \alpha_2(\q).\]
\end{lemma}

\begin{proof}
Assume $\q\in\qf^k$, where $k\in \{2,3\}$. Observe that for $x\in f^{-1}(\infty)\sm\{\infty\}$, the $1$-form $\q\otimes f^*\btheta_k$ is holomorphic at $x$. Indeed,  $\btheta_k$ vanishes at $\infty$, $f^*\btheta_k$ vanishes at $x$, and $\q$ has at worst a simple pole at $x$. Therefore, 
\begin{eqnarray*}
\alpha_k(f_*\q )&:=&\res\bigl((f_*\q)\otimes \btheta_k,\infty\bigr)\\
&\underset{\text{Lemma \ref{lem:pushpull}}}=&\sum_{x\in f^{-1}(\infty)} \res(\q\otimes f^*\btheta_k,x)=\res(\q\otimes f^*\btheta_k,\infty).
\end{eqnarray*}

\noindent{\bf Case 1.} If $k=3$, then $\btheta_3$ vanishes to order $2$ at $\infty$. It follows from Lemma \ref{multiplierlemma} below that $f^*\btheta_3- \mu_\infty \btheta_3$ vanishes to order $3$ at $\infty$. 
Since $\q$ has at worst a triple pole at $\infty$, we have 
\[\alpha_3(f_*\q)=\res(\q\otimes f^*\btheta_3,\infty)= \res(\q\otimes \mu_\infty \btheta_3,\infty)=\mu_\infty\alpha_3(\q).\]

\noindent{\bf Case 2.} If $k=2$, then $\btheta_2$ vanishes to order $1$ at $\infty$. It follows from Lemma \ref{multiplierlemma} below that $\ds f^*\btheta_2 - \frac{1}{d_\infty} \btheta_2$
vanishes to order $2$ at $\infty$. Since $\q$ has at worst a double pole at $\infty$, we have 
\[\alpha_2(f_*\q)=\res(\q\otimes f^*\btheta_2,\infty)= \res\left(\q\otimes \frac{1}{d_\infty}\btheta_2,\infty\right)=\frac{1}{d_\infty}\alpha_2(\q).\qedhere\]
\end{proof}

\begin{lemma}\label{multiplierlemma}
Let $f$ be a holomorphic map fixing a point $x$ with multiplier $\mu$ and local degree $d$. Let $\btheta$ be a holomorphic vector field vanishing at $x$ with order $m$.
\begin{itemize}
\item  If $d=1$, then $f^*\btheta-\mu^{m-1}\btheta$ vanishes to order $m+1$ at $x$. 
\item If $m=1$, then $f^*\btheta-\frac{1}{d}\btheta$ vanishes to order $m+1$ at $x$.
\item  If $d\geq 2$ and $m\geq 2$, then $f^*\btheta$ vanishes to order $m+1$ at $x$. 
\end{itemize}
\end{lemma}

\begin{proof}
Let $\zeta$ be a local coordinate vanishing at $x$. We may write 
\[\zeta\circ f = a\zeta^d\cdot \bigl(1+{\cal O}(\zeta)\bigr)\and \btheta =  b\zeta^m \frac{\d}{\d\zeta}\cdot \bigl(1+ {\cal O}(\zeta)\bigr)\]
with $a\neq 0$ and $b\neq 0$. In addition, if $d=1$, then $a=\mu$. 
Then, 
\[f^*\btheta =  \frac{b a^m \zeta^{dm}}{da\zeta^{d-1}} \frac{\d}{\d \zeta} \cdot \bigl(1+{\cal O}(\zeta)\bigr) = 
\frac{a^{m-1}}{d} \zeta^{(d-1)(m-1)} \btheta\cdot \bigl(1+{\cal O}(\zeta)\bigr).\qedhere\] 
\end{proof}

\subsection{The matrix $A_f$}\label{coefficients}

\begin{lemma}\label{lem:ayx}
The matrix of $f_*:\qf^3\to \qf^3$ in the basis $\{\q_x\}_{x\in\pf\sm\{\infty\}}$ is  $A_f$.
\end{lemma}

\begin{proof}
Since $\{\q_y\}_{y\in \pf\sm\{\infty\}}$ forms a basis of $\qf^3$, we may write
\[
f_*\q_x=\sum_{y\in\pf\sm\{\infty\}} f_{y,x}\cdot \q_y.
\]
We need to show that $f_{y,x} = a_{y,x}$ for all $x,y\in \pf\sm\{\infty\}$. We shall apply Lemma \ref{lem:pushpull} with 
\[
\q:=\q_x\quad  \text{and}\quad  \btheta:=\frac{\d }{\d z}.
\]
Fix $y_0\in\pf\sm\{\infty\}$. Note that $\btheta$ is holomorphic at $y_0$, and for $y\neq y_0$, $\q_y$ is holomorphic at $y_0$. In addition, $\ds \q_{y_0}\otimes\btheta = \frac{\d z}{z-y_0}$. Therefore 
\begin{eqnarray*}
f_{y_0,x}&=&\res\bigl((f_*\q_x)\otimes\btheta,y_0\bigr)\\
&\underset{\text{Lemma \ref{lem:pushpull}}}=&\sum_{w\in f^{-1}(y_0)}\res(\q_x\otimes f^*\btheta,w)\\
&=&\sum_{ {w\in  f^{-1}(y_0)\cap (\cf\cup\{x\})}} \res\left(\frac{\d z}{(z-x)f'(z)},w\right) = a_{y_0,x}.
\end{eqnarray*}
For the third equality, we use that 
\[
\q_x\otimes f^*\btheta=\frac{\d z}{(z-x)f'(z)}
\]
is holomorphic outside of $\cf\cup\{x\}$.
\end{proof}

\section{Periodic unicritical polynomials}\label{sec:unicritical}

\noindent From now on, we shall restrict our study to the case of unicritical polynomials. 
Any such polynomial is conjugate by an affine map to a polynomial of the form 
\[f_c(z):=z^D+c\quad\text{with}\quad c\in \C.\]
The map $f_c$ has a critical point at $z_0=0$ and a critical value at $f_c(0)=c$. 

Such a polynomial is postcritically finite if and only if the critical point $0$ is either periodic or preperiodic. We will restrict our study to the periodic case, and abusing terminology, we shall say that {\em $f_c$ is periodic of period $m$}  if $0$ is periodic of period $m$ for $f_c$. 

\subsection{Families of unicritical polynomials}

Before studying the corresponding sets $\spec_{f_c}$, we introduce some subsets of parameter space. 
The Multibrot set $\M_D$ is defined as 
\[\M_D:=\bigl\{c\in \C~|~\text{the sequence }\bigl(f_c^{\circ n}(0)\bigr)_{n\geq 1}\text{ is bounded}\bigr\}.\]
The set $\M_2$ the set is called the {\em Mandelbrot set}. 

\begin{figure}[htbp]
\centerline{
\includegraphics[width=5cm]{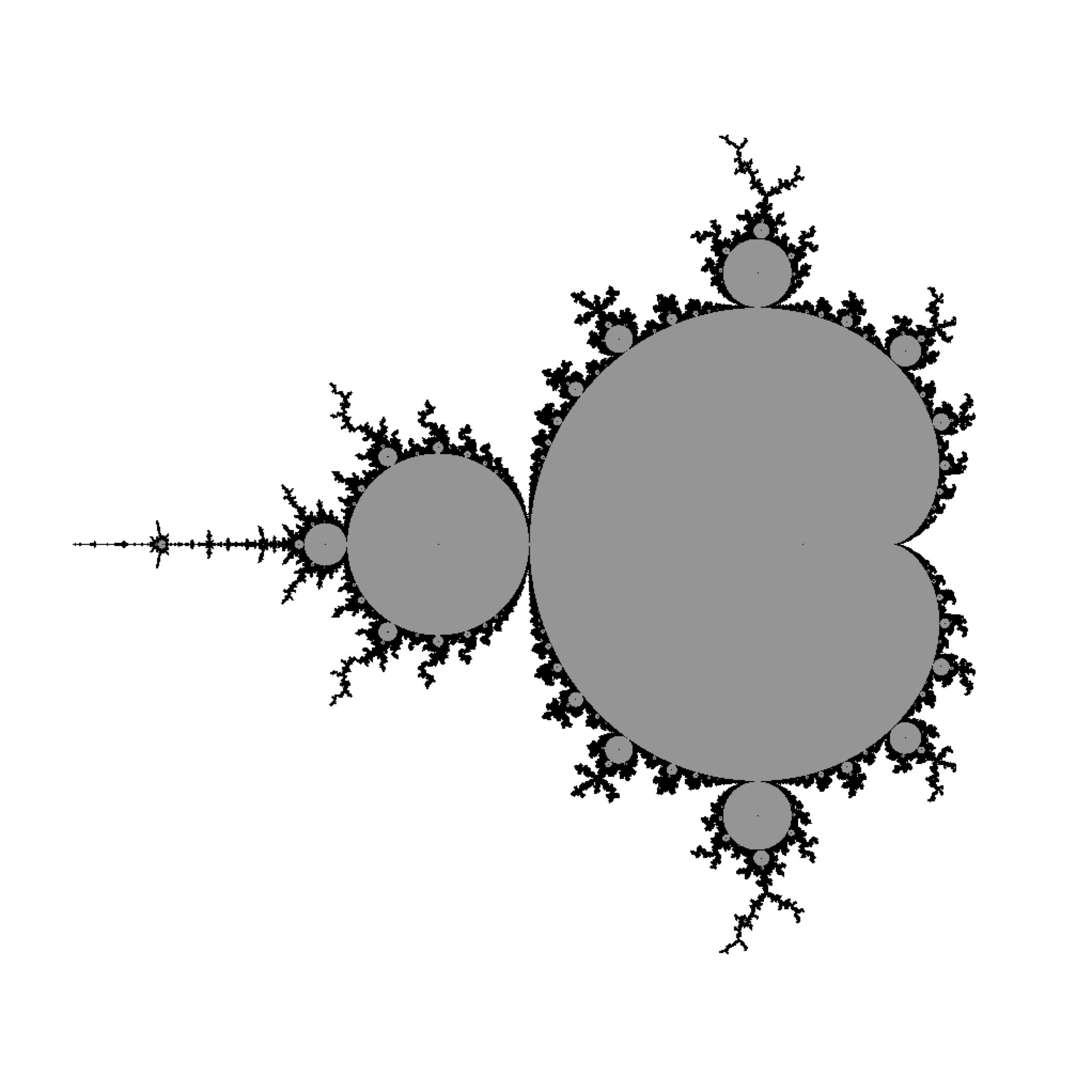}
\includegraphics[width=5cm]{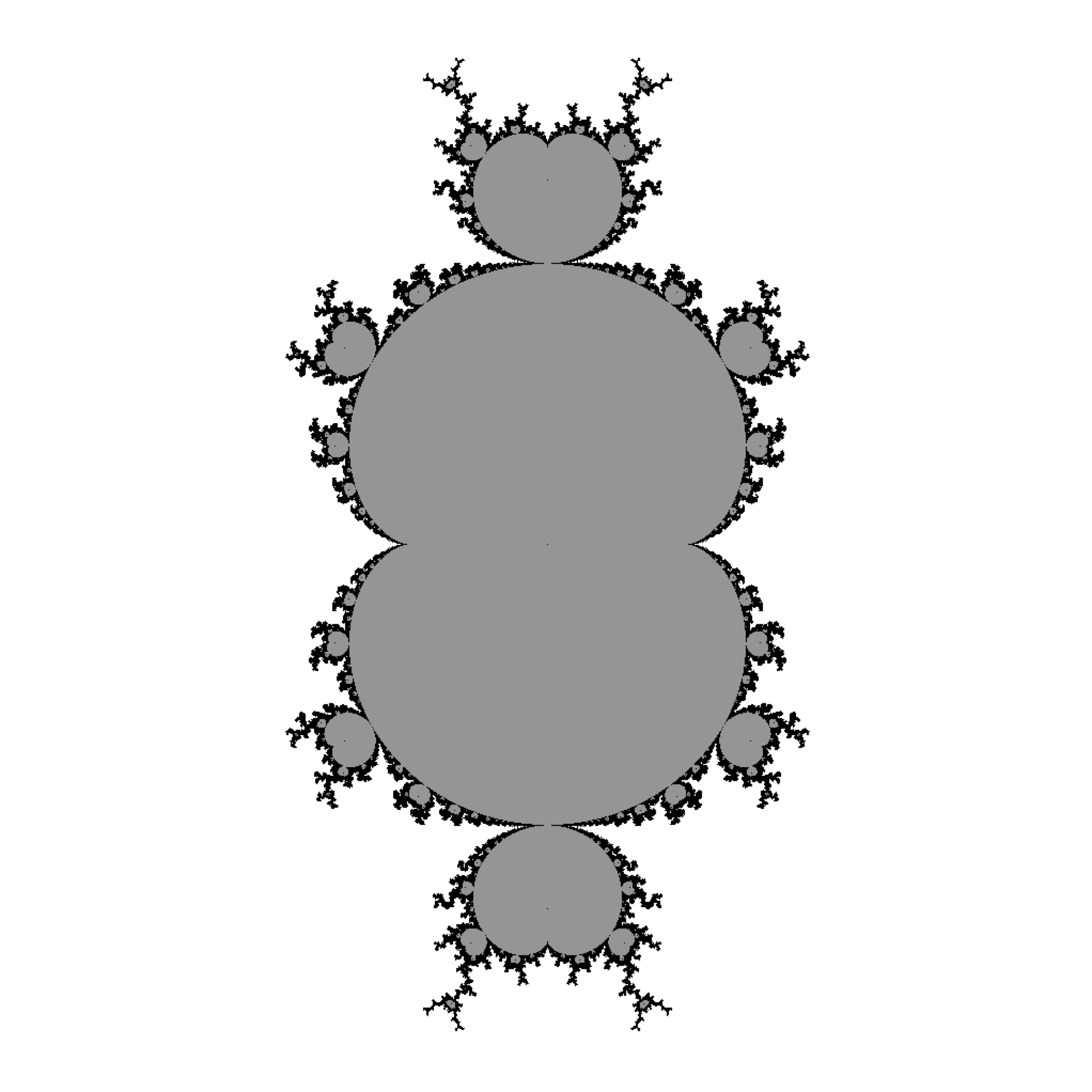}
}
\caption{Left: the Mandelbrot set. Right: the Multibrot set $\M_3$.\label{fig:mandels}}
\end{figure}

Note that the set of parameters $c\in \C$ such that $f_c$ has a periodic critical point is contained in the interior of $\M_D$. 
In addition, each component $U$ of the interior of $\M_D$ contains at most one parameter $c$ such that $f_c$ is periodic; in that case, $c$ is called the center of $U$. The boundary of $\M_D$ is contained in the closure of the set of centers. 

Observe that $\M_D$ has a symmetry of order $D-1$. Indeed, for $c\neq 0$, the linear map $z\mapsto w :=z/c$ conjugates the polynomial $f_c$ to the polynomial 
\[g_b: w\mapsto b w^D + 1\quad\text{with} \quad b:=c^{D-1}.\]
The polynomial $g_b$ has a critical point at $0$ and a critical value at $g_b(0)=1$. 
We set 
\[\MM_D:= \{b\in \C~|~\text{the sequence }\bigl(g_b^{\circ n}(0)\bigr)_{n\geq 1}\text{ is bounded}\bigr\}.\]
Then, $c\in \M_D$ if and only if $c^{D-1}\in \MM_D$. 

\begin{figure}[htbp]
\centerline{
\includegraphics[width=6cm]{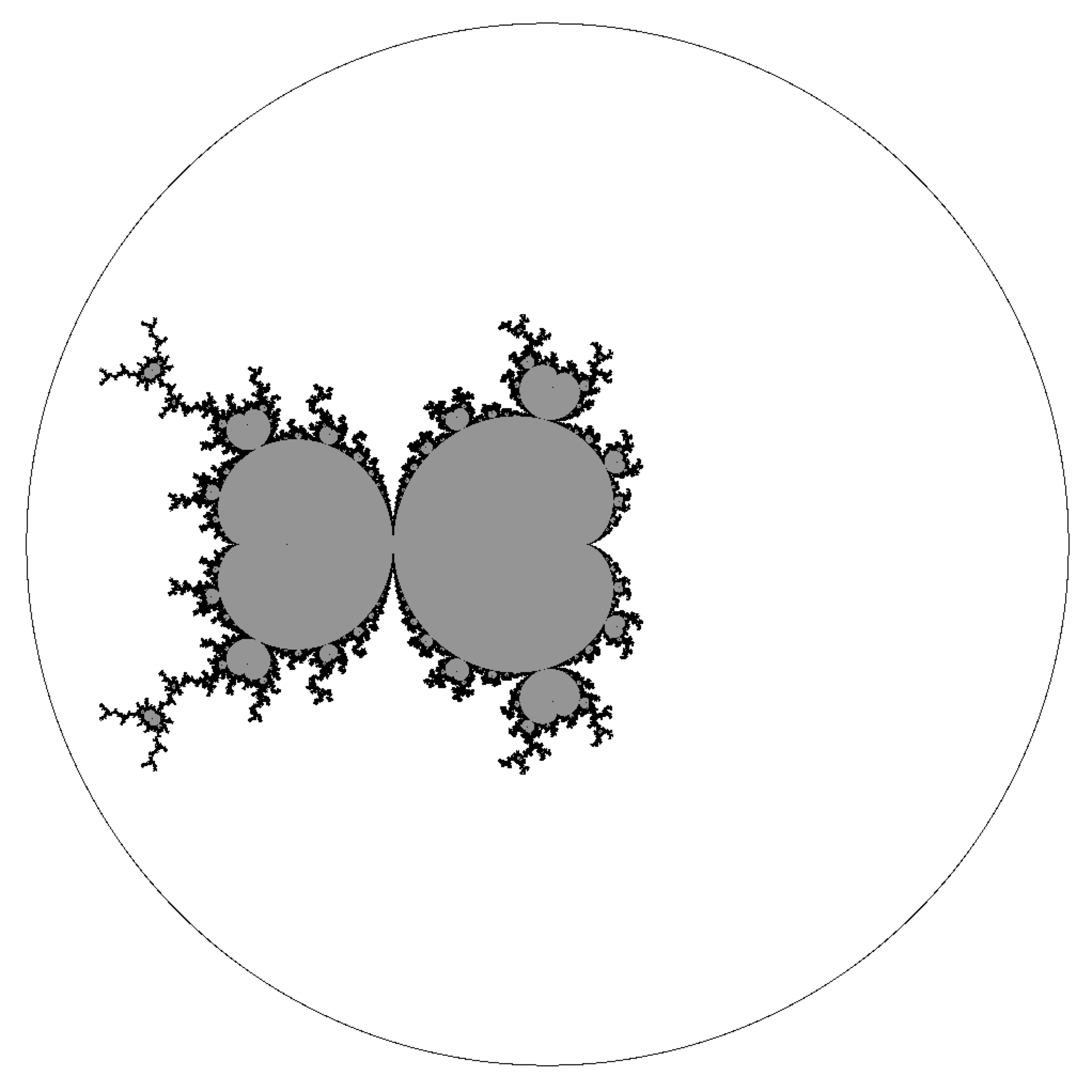}\qquad
\includegraphics[width=6cm]{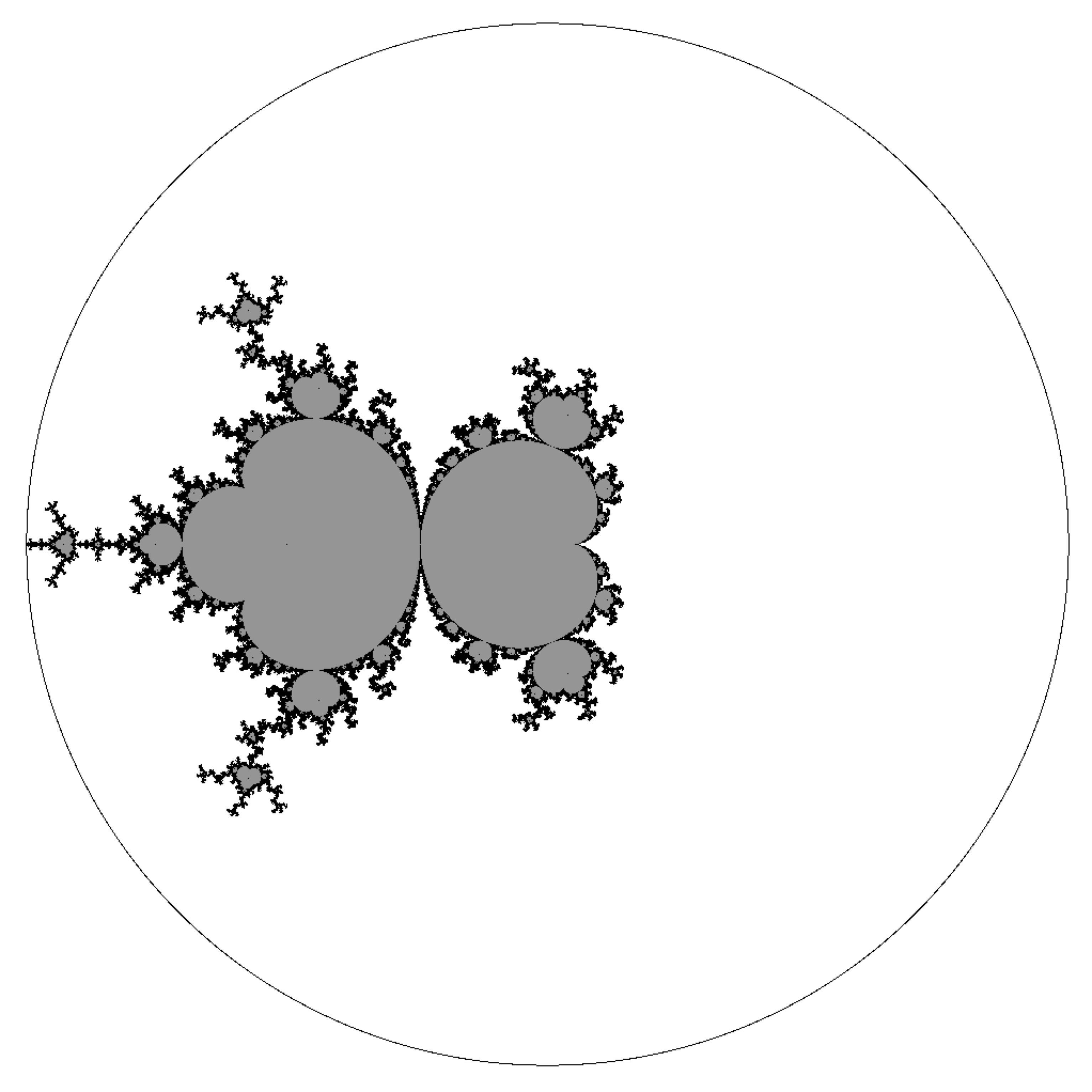}
}
\caption{Left: the set $\MM_3$. Right: the set $\MM_4$. Both are contained in the closed disk $|b|\leq 2$.  \label{fig:nandels}}
\end{figure}

\begin{lemma}
If $b\in \MM_D$, then $|b|<2$.
\end{lemma}

\begin{proof}
Assume $|b|>2$ and set $\kappa := |b|-1>1$. Set $z_n:=g_b^{\circ n}(0)$. Let us prove by induction on $n\geq 1$ 
that $|z_{n+1}|\geq \kappa|z_n|>1$. 
This is true for $n=1$ since $z_1=1$ and 
\[|z_2| = |b+1|\geq |b|-1 = \kappa |z_1|> 1.\]
If the property holds for some $n\geq 1$, then 
\[|z_{n+1}| = | bz_n^D+1|\geq |b| |z_n|^D-1 \geq |b||z_n|-|z_n| = \kappa|z_n|>1.\]
As a consequence, when $|b|>1$, the sequence $(z_n)_{n\geq 1}$ diverges and $b\not\in \MM_D$.
\end{proof}

\begin{corollary}\label{coro:MD}
 If $c\in \M_D$, then $|c|^{D-1}\leq 2$. 
 \end{corollary}
 
If $f:\C\to \C$ is a polynomial, the filled-in Julia set $K_f$ is defined as 
\[K_f:=\bigl\{z\in \C~|~\text{the sequence }\bigl(f^{\circ n}(z)\bigr)_{n\geq 1}\text{ is bounded}\bigr\}.\]

\begin{lemma}\label{lem:boundkf}
If $c\in \M_D$ and $z\in K_{f_c}$, then $|z|^{D-1}\leq 2$.  
\end{lemma}
 
\begin{proof}
Assume $c\in \M_D$ and $|z|^{D-1}>2$. For $n\geq 0$, set 
\[z_n:=f_c^{\circ n}(z)\and \kappa:=|z|^{D-1}-1>1.\] 
We will prove by induction on $n\geq 0$ that  $|z_n|>|c|$ and $|z_{n+1}|\geq \kappa |z_n|$.

According to the previous corollary, $|z_0|^{D-1}>2\geq |c|^{D-1}$. Thus, $|z_0|>|c|$. 
In addition, 
\[|z_{1}| = |z_0^{D}+c|\geq |z_0|^D-|c|>|z_0|\cdot |z_0|^{D-1}-|z_0| = \kappa|z_0|.\]
So, the property holds for $n=0$. 

Now, if the property holds for some $n\geq 0$, we may write 
\[|z_{n+1}|\geq \kappa|z_n|\geq \kappa|c|>|c|\]
and
\[|z_{n+2}| = |z_{n+1}^{D}+c|\geq |z_{n+1}|^D-|c|>|z_{n+1}|\cdot |z_{n+1}|^{D-1}-|z_{n+1}|\geq \kappa |z_{n+1}|,\]
so that the property holds for $n+1$. 

Thus, $|z_n|\geq \kappa^n|z_0|$, the sequence $(z_n)_{n\geq 1}$ is not bounded and $z\not\in K_{f_c}$. 
 \end{proof}

\subsection{Gleason polynomials}\label{sec:gleason}

For $n\geq 0$, let $G_n\in \Z[c]$ be defined by 
\[G_n(c) := f_c^{\circ n}(0).\]
Alternatively, the polynomials $G_n\in \Z[c]$ may be defined recursively by: 
\[G_0:=0\and G_n(c):=G_{n-1}^D(c)+c.\] 
Then, $G_n(c)=0$ if and only if $c$ is a center of period $p$ dividing $n$. 

\begin{example}
For $D=2$, we have 
\[G_1(c) = c,\quad G_2 (c)= c^2+c, \quad G_3(c) = c^4+2c^3+c^2+c.\]
\end{example}

Note that $G_n$ is a monic polynomial. 

\begin{lemma}[Gleason]\label{lem:gleason}
For each $n\geq 0$, the polynomial $G_n$ has simple roots. 
\end{lemma}

\begin{proof}
For each $n\geq 1$,
\[G'_n = DG_{n-1}^{D-1}G'_{n-1} + 1 \equiv 1 \mod D.\] 
In addition, $G_n$ is monic, and so, the resultant of $G_n$ and $G'_n$ is equal to $1 \mod D$. In particular, it does not vanish. 
\end{proof}

As a corollary, we deduce that for $m\geq 1$, there exists a (unique) monic 
polynomial $H_m\in \Z[c]$, 
such that for $n\geq 1$
\[G_n = \prod_{m|n} H_m.\]
The roots of $H_m$ are exactly the centers of period $m$.
The monomial of least degree of $G_n(c)$ is $c$ with coefficient $1$. Since $H_1(c)=c$, we see that for $m\geq 2$ the constant coefficient of $H_m$ is $1$.  In particular, with the exception of $c=0$, centers are algebraic units. 

\begin{example}
For $D=2$, we have $G_4 = H_1H_2H_4$ with 
\[H_1 (c)= c, \quad H_2 (c)= c+1\and H_4 (c)= c^6+3c^5+3c^4+3c^3+2c^2+1.\]
\end{example}

We shall also use the following result due to Bjorn Poonen. 

\begin{lemma}[Poonen]\label{lem:resgleason}
For $m\neq n$, $\result(H_m,H_n)=\pm 1$. 
\end{lemma}

\begin{proof} Assume $n>m$. 
It is not hard to see by induction on $k\geq 1$, that 
\[G_{m+k}\equiv G_k\mod G_m^D.\]
This implies that, $G_{m n}\equiv G_m \mod G_m^D$. Since $m<n$,  $G_m H_n$ divides $G_{mn}$. So, there are polynomials $A\in \Z[c]$ and $B\in \Z[c]$ such that 
\[ AG_m H_n = G_{mn} = G_m + B G_m^D.\]
Dividing by $G_m$ yields $AH_n-BG_m^{D-1} = 1$. It follows that $H_n$ and $H_m$ are relatively prime in $\Z[c]$ and $\result(H_m,H_n)=\pm 1$. 
\end{proof}

\begin{remark}
For $m\geq 2$, we have $H_m(c) = J_m(c^{D-1})$ for some polynomial $J_m\in \Z[b]$. It might be tempting to conjecture that for all $D\geq 2$ and all $m\geq 2$, the polynomial $J_m$ is irreducible. However for $D=7$, $J_3$ is reducible: 
\[J_3(b) = (b^2+b+1)(b^6+6b^5+14b^4+15b^3+6b^2+1).\]
This is further discussed in \cite{Xnote}. 
\end{remark}

\subsection{Periodic points}

\begin{proposition}\label{prop:largeeig}
Let $f$ be a periodic unicritical polynomial of degree $D$. If $\lambda\in \Lambda_f\sm\{0\}$, then $D\lambda$ is an algebraic unit and $\frac{1}{2D} \leq |\lambda|< 1$. 
\end{proposition}

\begin{proof}
According to Proposition \ref{prop:eigmult},  $\lambda\in \Lambda_f\sm\{0\}$ if and only if $1/\lambda^n$ is the 
multiplier of a cycle of period $n$ not contained in $\pf$, i.e. a repelling cycle of $f_c$. In particular, $|\lambda|<1$. 

Assume the critical point of $f$ is periodic of period $m$. Conjugating $f$ with an affine map, we may assume that 
\[f(z) = f_c(z) := z^D+c\quad\text{with}\quad H_m(c)=0.\] 
Now, let $z_1\mapsto z_2\mapsto \cdots \mapsto z_n\mapsto z_1$ be a cycle of period $n$. The multiplier of the cycle is
\[\mu = D^n (z_1z_2\cdots z_n)^{D-1}.\]
According to Lemma \ref{lem:boundkf}, $|z_j|^{D-1}\leq 2$, so that $|\mu|\leq (2D)^n$. Thus, it suffices to prove that the points $z_j$ are algebraic units. 

Let us first assume that $m$ does not divide $n$, so that $f_c^{\circ n}(0)\neq 0$. 
The points $z_j$ are roots of the polynomial $f_c^{\circ n}(z)-z\in \Z[c,z]$. 
Denote by $R_z$ the polynomial $f_c^{\circ n}(z)-z$ considered as a polynomial of the variable $c$ with coefficients 
in $\Z[z]$ and set 
\[S(z):=\result\bigl(H_m,R_z\bigr) = \prod_{H_m(c)=0} \bigl(f_c^{\circ n}(z)-z\bigr).\]
Note that, as a product of monic polynomials, $S$ is a monic polynomial. In addition, $f_c^{\circ n}(0) = G_n(c)$ and the constant coefficient of $S$ is 
\[S(0) =  \prod_{H_m(c)=0} \bigl(f_c^{\circ n}(0)-0\bigr) =  \prod_{H_m(c)=0} G_n(c) =   \result (H_m,G_n).\] 
According to Lemma \ref{lem:resgleason}, since $m$ does not divide $n$, this resultant  is equal to $±1$.
This shows that the points $z_j$ are algebraic units. 

Let us now assume that  $m$ divides $n$. Then, the constant coefficient of $f_c^{\circ n}(z)-z$ vanishes and 
\[g_c(z):=\frac{f_c^{\circ n}(z) - z }{z}\]
is a monic polynomial  with constant coefficient $g_c(0) = -1$. 
Denote by $R_z$ the polynomial $g_c(z)$ considered as a polynomial of the variable $c$ with coefficients 
in $\Z[z]$ and set 
\[S(z):=\result\bigl(H_m,R_z\bigr) = \prod_{H_m(c)=0} g_c(z).\]
Again, $S$ is a monic polynomial and its constant coefficient is 
\[S(0) =  \prod_{H_m(c)=0} g_c(0) = ±1.\]
Thus, the points $z_j$ are algebraic units. 
\end{proof}

\subsection{Characteristic polynomials}

We assume that $c$ is a center of period $m\geq 3$ and  $f:=f_c$, so that $\dim(\qf) = m-2\geq 1$. 

Note that $f(\infty) = \infty\in \pf$. We denote by $\chi_f$ the characteristic polynomial of $f_*:\qf\to \qf$ and for $k\in \{2,3\}$, we denote by $\chi_f^k$ the characteristic polynomial of $f_*:\qf^k\to \qf^k$. The local degree of $f$ at $\infty$ is $D$ and the multiplier of $f$ at $\infty$ is $0$. 
According to \S\ref{sec:charac}, we have 
\[\chi_f^3(\lambda) = \lambda\cdot \chi_f^2 (\lambda)= \lambda\left(\lambda-\frac{1}{D}\right)\chi_f(\lambda).\]

For $n\in \Z/m\Z$, set 
\[\zeta_n:=f^{\circ n}(0),\quad \delta_n:= f'(\zeta_n) = D \zeta_n^{D-1}.\]
Set 
\[\Delta_1 := \delta_1,\quad\Delta_2:=\delta_1\delta_2,\quad\ldots,\quad  \Delta_{m-1} := \delta_1\delta_2\cdots \delta_{m-1}.\]

\begin{proposition}
We have that 
\[\chi_f^2(\lambda) = \lambda^{m-1} + \frac{1}{\Delta_1}\lambda^{m-2} + \cdots + \frac{1}{\Delta_{m-2}} \lambda + \frac{1}{\Delta_{m-1}}.\]
\end{proposition}

\begin{proof}
For $n\in \Z/m\Z$, set 
\[\q_n:=\frac{(\dz)^2}{z-\zeta_n}.\]
The matrix $A_f$ of $f_*:\qf^3\to \qf^3$ in the basis $\{\q_n\}_{n\in \Z/m\Z}$ is provided by \S\ref{sec:charac}. 
We have 
\[
f_*\q_n=\begin{cases}
0&\text{if }n=0\text{ and}\\
\ds \frac{1}{\delta_n}\q_{n+1}-\frac{1}{\delta_n}\q_1& \text{if } n\neq 0.
\end{cases}
\]
and so 
\[A_f:=\begin{bmatrix}
0 & 0 & 0& \cdots & 0 &a_{m-1}\\
0 &- a_1&- a_2 & \cdots  & -a_{m-2} & -a_{m-1}\\
0& a_1 & 0 &0 & \cdots & 0\\
\vdots & 0 & a_2 & \ddots & \ddots& \vdots\\
\vdots & &\ddots&\ddots & 0 & 0\\
0 & \cdots & \cdots & 0 &a_{m-2} & 0
\end{bmatrix}\quad \text{with}\quad a_n:=\frac{1}{\delta_n}.\]
The characteristic polynomial of $A_f$ is $\chi_f^3 (\lambda)= \lambda\chi_f^2(\lambda)$, so that 
\begin{align*}
\chi_f^2(\lambda) &= (\lambda+a_1)\lambda^{m-2}+a_1a_2\bigl(\lambda^{m-3} + a_3(\lambda^{m-4}+\cdots)\bigr)\\
&= \lambda^{m-1} + \frac{1}{\Delta_1}\lambda^{m-2} + \cdots + \frac{1}{\Delta_{m-2}} \lambda + \frac{1}{\Delta_{m-1}}.\qedhere
\end{align*}
\end{proof}

It shall be convenient to work with the polynomial $\chi_f^2/\chi_f^2(0)$
which has the same zeros as $\chi_f^2$. 
Setting 
\[\Delta_0:=1,\quad \Delta_{-1} := \delta_{m-1}, \quad \Delta_{-2} := \delta_{m-1}\delta_{m-2},\quad \ldots,\quad \Delta_{-(m-1)} := \delta_{m-1}\delta_{m-2}\cdots \delta_{1},\]
we get that for $n\in \ob1,m-1\cb$, 
\[\Delta_n\cdot \Delta_{-(m-1-n)} = \Delta_{m-1} = \frac{1}{\chi_f^2(0)},\] so that 
\[\frac{\chi_f^2(\lambda)}{\chi_f^2(0)} = 1+\Delta_{-1}\lambda+\Delta_{-2}\lambda^2+\cdots+\Delta_{-(m-1)}\lambda^{m-1}.\]

The polynomials $G_n\in \Z[c]$ defined by $G_n(c) = f_c^{\circ n}(0)$ have simple roots (see Lemma \ref{lem:gleason}).  This is related to the fact that for a postcritically finite rational map $f$ which is not a flexible Lattès map, $1\not \in \spec_f$ (see Proposition \ref{prop:noeigen1}).

\begin{proposition}
If $c$ is a center of period $m\geq 3$ and $f:=f_c$, then 
\[G'_m(c) = (1-D) \frac{\chi_f(1)}{\chi_f(0)}.\]
\end{proposition}

\begin{proof}
We have $G_1 (c)= c$ and $G_{n+1} = G_n^D(c)+c$, so that 
\[G_m' = 1+D\ G_{m-1}^{D-1} G_{m-1}' = 1 + \delta_{m-1} G_{m-1}' .\]
Since $G_1' = 1$, we have 
\begin{align*}
G_m'(c)  &= 1 + \delta_{m-1}\cdot \bigl(1+\delta_{m-2}\cdot (1+\cdots(1+\delta_1))\bigr) \\
&= 1 + \delta_{m-1} + \delta_{m-1}\delta_{m-2} + \cdots + \delta_{m-1}\delta_{m-2}\cdots \delta_1\\
& = 1 + \Delta_{-1} + \Delta_{-2}+\cdots +\Delta_{-(m-1)} = \chi_f^2(1)/\chi_f^2(0).
\end{align*}
Now, 
\[\chi_f^2(0) = -\frac{1}{D} \chi_f(0)\quad \text{and} \quad \chi_f^2(1) = \left(1-\frac{1}{D}\right) \chi_f(1)\]
so that 
\[ \frac{\chi_f^2(1)}{\chi_f^2(0)} = (1-D)  \frac{\chi_f(1)}{\chi_f(0)}.\qedhere\]
\end{proof}

\subsection{Algebraic units}\label{sec:units}

\begin{theorem}\label{theo:units}
Let $f$ be a periodic unicritical polynomial of degree $D$. If  $\lambda$ is an eigenvalue of $f_*:\qf\to \qf$, then $D\lambda$ is an algebraic unit. 
\end{theorem}

\begin{proof}
The polynomial $f$ is conjugate to some polynomial $f_c$ where $c$ is a center of period $m$; that is, $H_m(c)=0$. 
The eigenvalues of $f_*$ are those of $(f_c)_*$. 

For $n\in \ob0,m-1\cb$, let $\Gamma_n\in \Z[c]$ be the polynomial defined by 
\[\Gamma_n:= \prod_{k=1}^n G_{m-k}^{D-1},\]
where as usual, an empty product is equal to $1$. 
Then 
\[\Delta_{-n} = D^n\prod_{k=1}^n \zeta_n^{D-1} = D^n\prod_{k=1}^n \bigl(f_c^{\circ n}(0)\bigr)^{D-1} =  D^n \Gamma_n(c).\]
As a consequence, 
\[\frac{\chi^2_{f_c}(\lambda)}{\chi^2_{f_c}(0)} = \sum_{n=0}^{m-1} D^n \Gamma_n(c) \lambda^n = R(c,D \lambda),\]
where $R\in \Z[c,\nu]$ is defined by 
\[R(c,\nu):=\sum_{n=0}^{m-1} \Gamma_n(c) \nu^n.\]
We shall denote by $R_\nu$ the polynomial $R(c,\nu)$ considered as a polynomial in the variable $c$ with coefficients in $\Z[\nu]$. Let $S_m\in \Z[\nu]$ be defined by 
\begin{equation}\label{eq:sm}
S_m(\nu) := \result(H_m,R_\nu) = \prod_{H_m(c)=0}R(c,\nu) = \prod_{H_m(c)=0} \frac{\chi^2_{f_c}(\nu/D)}{\chi^2_{f_c}(0)}.
\end{equation}
If $\lambda\in \spec_{f_c}$ with $H_m(c)=0$, then $\nu:=D \lambda$ is a root of $S_m$. 

On the one hand, the constant coefficient of $S_m$ is $S_m(0)=1$. On the other hand, the leading monomial of $R_m(c,\nu)$ considered as a polynomial of $\nu$ is $\Gamma_{m-1}(c)\nu^{m-1}$, so that the leading coefficient of $S_m$ is:
\[\prod_{H_m(c)=0}\Gamma_{m-1}(c) = \result(H_m,\Gamma_{m-1}) = \result\left(H_m,\prod_{k=1}^{m-1} G_{m-k}^{D-1}\right).\]
By Lemma \ref{lem:resgleason}, this resultant is equal to $\pm 1$. It follows that the roots of $S_m$ are algebraic units. 
\end{proof}

We know that $\chi_{f_c}^2(\lambda) = (\lambda-1/D)\chi_{f_c}(\lambda)$, so that the factor $(\nu-1)$ appears $\deg(H_m)$ times in the product \eqref{eq:sm} defining $S_m$. In addition, if $c_1^{D-1}=c_2^{D-1}$, then the polynomials $f_{c_1}$ and $f_{c_2}$ are conjugate. So, each eigenvalue is counted $D-1$ times in the product. Thus, there is a (unique) polynomial $\Upsilon_m\in \Z[\nu]$ with constant coefficient $1$ such  that 
\[S_m (\nu) = (1-\nu)^{\deg(H_m)}\cdot \Upsilon_m^{D-1}(\nu).\] 
It would be interesting to prove that the roots of $\Upsilon_m$ are simple. For example, when this is the case,
$f_*:\qf\to \qf$ is diagonalizable. 

\subsection{The case $m=3$}

We will prove that the roots of $\Upsilon_3$ are simple. 

\begin{lemma}
For all $D\geq 2$, 
\[\Upsilon_3 (\nu)=(\nu+1)^{D+1} - \nu^D.\]
\end{lemma}

\begin{proof}
We have that $G_3(c) = (c^D+c)^D+c$, so that 
\[H_3(c) = A(c^{D-1})\quad \text{with}\quad A(b) = b(b+1)^D+1.\]
In addition, an elementary computation shows that 
\[R(c,\nu) = (1-\nu)B_\nu(c^{D-1})\quad\text{with}\quad B_\nu (b) = 1-\nu b^2(b+1)^{D-1}.\]
Then, 
\[\Upsilon_3 (\nu)= \result(A,B_\nu)\]
Observe that 
\[\nu bA(b) + (b+1)B_\nu(b) = (\nu+1)b+1.\]
Note that $A$ is monic of degree $D+1$. 
We therefore have
\begin{align*}
(\nu+1)^{D+1}\cdot \result\left(A(b), b+\frac{1}{\nu+1}\right)
&=\result\bigl(A(b),(\nu+1) b+1 \bigr) \\
&= \result\bigl(A(b),(b+1)B_\nu(b)\bigr)\\
&=\result\bigl(A(b),b+1\bigr)\cdot \Upsilon_3(\nu)
\end{align*}
As a consequence, 
\[\Upsilon_3 (\nu)= (\nu+1)^{D+1}\cdot  \frac{A\bigl(-1/(\nu+1)\bigr)}{A(-1)} = (\nu+1)^{D+1} - \nu^D.\qedhere\]
\end{proof}

\begin{lemma}
For all $D\geq 2$, the roots of $\Upsilon_3$ are simple. 
\end{lemma}

\begin{proof}
Note that $\nu\Upsilon'_3 (\nu)-D\Upsilon_3(\nu) = (1+\nu)^D(\nu-D)$. So, if $\Upsilon_3$ and $\Upsilon'_3$ had a common root, this would be either $-1$ or $D$. None of those are roots of $\Upsilon_3$. 
\end{proof}

\subsection{Spectral gap}\label{sec:gap}

For $m\geq 3$, let $\spec(D,m)$ be the union of the sets of eigenvalues $\spec_f$ for all unicritical polynomials $f$ of degree $D$ which are periodic of period $m$. Set 
\[\spec(D):=\bigcup_{m\geq 4} \spec(D,m).\]
The following theorem is illustrated by Figure \ref{fig:bound}.

The following proposition completes the proof of Theorem \ref{theo:3}. 

\begin{proposition}[Spectral gap]\label{theo:gap}
If $\lambda\in \spec(D)$, then 
\[
\frac{1}{4D}<|\lambda|<1.
\]
\end{proposition}

\begin{proof}
Assume $c$ is a center of period $m$ and $f=f_c$. Assume $\lambda\in \spec_f$. 
According to Corollary \ref{coro:lambdainDforpoly}, we have that $|\lambda|<1$. We must show that $|\lambda|>1/(4D)$. 

Let us recall that 
\[0 = \frac{\chi^2_f(\lambda)}{\chi^2_f(0)}  = 1+\Delta_{-1}\lambda+\Delta_{-2}\lambda^2+\cdots+\Delta_{-(m-1)}\lambda^{m-1}.\]
with
\[\Delta_0:=1,\quad \Delta_{-1} := \delta_{m-1}, \quad \Delta_{-2} := \delta_{m-1}\delta_{m-2},\quad \ldots,\quad \Delta_{-(m-1)} := \delta_{m-1}\delta_{m-2}\cdots \delta_{1}.\]

\begin{lemma}
For all $n\in \Z/m\Z$, we have that $|\delta_n|\leq 2D$.
\end{lemma}

\begin{proof}
Set $\zeta_n:=f_c^{\circ n}(0)$. Note that $c\in \M_D$ and $\zeta_n\in K_{f_c}$. 
According to Lemma \ref{lem:boundkf}, we have $|\zeta_n|^{D-1}\leq 2$ for all  $n\in \Z/m\Z$,  so
\[|\delta_n| = D|\zeta_n|^{D-1} \leq 2D.\qedhere\]
\end{proof}

If $|z|\leq \frac{1}{4D}$, then for all $k\in \ob1,m-1\cb$,
\[| \Delta_{-k}z^k|<\frac{(2D)^k}{(4D)^k} = \frac{1}{2^k},\] 
so 
\[\left|\frac{\chi^2_f(z)}{\chi^2_f(0)}\right|\geq 1-   \sum_{k=1}^{m-1} \frac{1}{2^k}>0.\]
Since $\chi^2_f(\lambda)=0$, we necessarily have $|\lambda|>\frac{1}{4D}$. 
\end{proof}

\subsection{Equidistribution}\label{sec:quidistribution}

This section is devoted to the proof of Theorem \ref{theo:4}. 


In order to prove this result, we will show that when $r\in [r_D,1]$, there exists a sequence of centers $c_n$ such that the roots of $\chi_{f_{c_n}}$ equidistribute on the circle $|z|=r$ as $n\to \infty$.  This means the following. 

Given a polynomial $P\in \C[z]$, we denote by $\m_P$ the probability measure
\[\m_P := \frac{1}{\deg(P)} \sum_{x\in \C} \ord_x(P)\]
where $\ord_x(P)$ is the order of vanishing of $P$ at $x$. 

Assume $(P_n\in \C[z])_{n\geq 0}$ is a sequence of polynomials. We say that as $n\to \infty$, the roots of $P_n$ equidistribute according to a probability measure $\m$ if the sequence of probability measures $(\m_{P_n})_{n\geq 0}$ converges weakly to $\m$. 


We say that as $n\to \infty$, the roots of $P_n$ equidistribute on a Euclidean circle if the roots of $P_n$ equidistribute according to the normalized $1$-dimensional Lebesgue measure on this circle. 

\begin{example}
As $n\to \infty$, the roots of the polynomials $z^n-1$ equidistribute on the unit circle $S^1:=\{|z|=1\}$. 
\end{example}

\begin{proposition}\label{approx_centers}
Suppose that the critical point of $f_{c_0}(z) = z^D+c_0$ is preperiodic to a repelling fixed point $\beta_0$ of multiplier $\mu$. Then, there exists a sequence of centers $c_n\in \M_D$ converging to $c_0$ such that as $n\to \infty$, the roots of $ \chi_{f_{c_n}}$ equidistribute on the circle $\bigl\{|\lambda| =1/|\mu|\bigr\}$.
\end{proposition}

\begin{remark}
A similar argument establishes that when the critical point of $f_{c}:z\mapsto z^D+c$ is preperiodic to a repelling cycle of multiplier $\mu$ and period $m$, then there exists a sequence of centers $c_n\in \M_D$ converging to $c$ such that  as $n\to \infty$, the roots of $ \chi_{f_{c_n}}$ equidistribute on the circle $\bigl\{|\lambda| = 1/\sqrt[m]{|\mu|}\bigr\}$. We will not use this fact. 
\end{remark}

\begin{proof}
Fix $r_1<r_0<|\beta_0|$ and $\epsilon>0$ so that $f_c$ has an inverse branch 
\[
g_c:D(\beta_0,r_0)\to D(\beta_0,r_1)
\]
for every $c\in D(c_0,\epsilon)$. 
For $c\in D(c_0,\epsilon)$, let $\beta(c)$ be the unique (repelling) fixed point of $f_c$ in $D(\beta_0,r_0)$. The map $\beta$ is holomorphic on $D(c_0,\epsilon)$ and $\beta(c_0)=\beta_0$. 

Choose a point $z_{-n_0}\in D(\beta_0,r_0)$ in the backward orbit of $0$ of $f_{c_0}$; that is, $f_{c_0}^{\circ n_0}(z_{-n_0})=0$ with $n_0>0$.  Since $0$ is not periodic (it is preperiodic to $\beta_0$), $(f^{\circ n_0}_{c_0})'(z_{-n_0})\neq 0$. Thus, taking $\epsilon>0$ closer to $0$ if necessary, we may assume that there is a holomorphic function $\zeta_{-n_0}:D(c_0,\epsilon)\to D(\beta_0,r_0)$ defined implicitly by $f_c^{\circ n_0}(\zeta_{-n_0}(c))=0$. 
Set
\[\zeta_j(c):=f_c^{\circ(n_0+j)}(\zeta_{-n_0}(c))\text{ for } j\geq -n_0\]
and 
\[\zeta_j(c):=g_c^{\circ(-n_0-j)}(\zeta_{-n_0}(c))\text{ for }j\leq -n_0.\]
We have specified a distinguished orbit for $f_c$
\[
\ldots\mapsto \zeta_{-2}(c)\mapsto \zeta_{-1}(c)\mapsto \zeta_0(c)=0\mapsto \zeta_1(c)=c\mapsto \zeta_2(c)\mapsto\ldots
\]
which converges to $\beta(c)$ in backward time.

Let $k_0$ be the preperiod of $0$ to $\beta_0$. As $n\to\infty$, the sequence $(\zeta_{k_0}-\zeta_{-n})$ converges uniformly on $D(c_0,\epsilon)$, to the function $\zeta_{k_0}-\beta$. This function vanishes at $c_0$, but it is not identically $0$. 
In particular, we can find a sequence $(c_n)$ such that $c_n\to c_0$ as $n\to \infty$, and $\zeta_{k_0}(c_n)=\zeta_{-n}(c_n)$. Then, $c_n$ is a center of period $m_n:=n+k_0$. 

\begin{lemma}\label{lemmabelow}
There exists a constant $K$ such that for all sufficiently large $n$ and all $1\leq j \leq m_n-1$
\[
\frac{1}{K}<\left|\frac{\Delta_j(c_n)}{\mu^j}\right|\leq K
\and 
\frac{1}{K}<\left|\frac{\Delta_{-j}(c_n)}{\mu^j}\right|\leq K
\]
\end{lemma}

\begin{proof}
For $j\in\Z$, let $\delta_j(c)$ be the derivative of $f_c$ at $\zeta_j(c)$.  Taking $\epsilon$ closer to $0$ if necessary, we may assume that there is a constant $K_0>0$ so that 
\begin{equation}\label{eq:k0}
\frac{1}{K_0}<\left|\frac{\delta_j(c)}{\mu}\right|<K_0
\end{equation}
for $c\in D(c_0,\epsilon)$ and $j\in \{-n_0,\ldots,-1\}\cup\{1,\ldots,k_0\}$. 

Let $d$ be the order of $\zeta_{k_0}-\beta$ at $c_0$ (it is true that $d=1$ \cite{orsay}, but we will not require this fact). Now there is a constant $K_1$ so that 
\begin{equation}\label{eq:k1}
K_1|c-c_0|^d<|\zeta_{k_0}(c)-\beta(c)|.
\end{equation}
By the Schwarz Lemma, there exists $\kappa<1$ such that for all $z\in D(\beta_0,r_0)$ and all $c\in D(c_0,\epsilon)$,
\[|g_c(z)-\beta(c)|<\kappa|z-\beta(c)|.\] 
Then, since $\zeta_{-n_0}(c)\in D(\beta_0,r_0)$ and $\zeta_{-n}(c) = g_c^{\circ (n-n_0)}\bigl( \zeta_{-n_0}(c)\bigr)$, we have
\begin{equation}\label{eq:kappa}
|\zeta_{-n}(c)-\beta(c)|<r_0\kappa^{n-n_0}
\end{equation}
Set 
\begin{equation}\label{eq:k2}
r_n:=\left(\frac{r_0\kappa^n}{K_1\kappa^{n_0}}\right)^{1/d}=K_2\kappa^{n/d}\quad \text{with}\quad K_2:=\left(\frac{r_0}{K_1\kappa^{n_0}}\right)^{1/d}.
\end{equation}
 For $c\in D(c_0,\epsilon)\sm D(c_0,r_n)$, so that $r_n\leq |c-c_0|$, we have
 \[|\zeta_{-n}(c)-\beta(c)|\underset{\eqref{eq:kappa}}<r_0\frac{\kappa^n}{\kappa^{n_0}}\underset{\eqref{eq:k2}}=K_1r_n^d\leq K_1|c-c_0|^d\underset{\eqref{eq:k1}}<|\zeta_{k_0}(c)-\beta(c)|,\] 
 so that the leftmost quantity cannot be equal to the rightmost quantity. As a consequence, 
 for $n$ large enough, 
 \[|c_n-c_0|<r_n.\]

For $j\leq -n_0$ and $c\in D(c_0,\epsilon)$, the point $\zeta_j(c)$ belongs to $D(\beta_0,r_0)$. Since
$\delta_j(c) = D\zeta_j^{D-1}(c)$, the branch of 
\[
c\mapsto \log\frac{\delta_j(c)}{\delta_j(c_0)}
\]
which vanishes at $c_0$ is bounded by some constant $K_3$. 
According to the Schwarz Lemma, 
\[
\left|\log\frac{\delta_j(c_n)}{\delta_j(c_0)}\right|<\frac{K_3r_n}{\epsilon}\underset{\eqref{eq:k2}}=K_4\kappa^{n/d}\quad\text{with}\quad K_4 = \frac{K_2K_3}{\epsilon}.
\] 
Then, for $-n\leq j_1< j_2\leq -n_0$ we have 
\[
\left|\log\frac{\delta_{j_1}(c_n)\cdots\delta_{j_2 - 1}(c_n)}{\delta_{j_1}(c_0)\cdots\delta_{j_2 - 1}(c_0)}\right|<nK_4\kappa^{n/d}\underset{n\to\infty}\longrightarrow 0.
\]
In addition, we have 
\[
\delta_{j_1}(c_0)\cdots\delta_{j_2 - 1}(c_0)=\mu^{j_2 - j_1}\frac{\phi'\bigl(\zeta_{j_1}(c_0)\bigr)}{\phi'\bigl(\zeta_{j_2}(c_0)\bigr)}
\]
where $\phi:D(\beta_0,r_0)\to\C$ is the linearizing map conjugating $f_{c_0}$ to multiplication by $\mu$. So, there is a constant $K_5$ such that for $n$ large enough and $-n\leq j_1< j_2\leq -n_0$,
\[
\frac{1}{K_5}<\left|\frac{\delta_{j_1}(c_n)\cdots\delta_{j_2 - 1}(c_n)}{\mu^{j_2 - j_1}}\right|<K_5.
\]
Using Inequality \eqref{eq:k0} and $\delta_j (c_n)= \delta_{m_n+j}(c_n)$ with $m_n = n+k_0$, we deduce that for $n$ large enough and $1\leq j_1< j_2\leq m_n-1$
\[
\frac{1}{K}<\left|\frac{\delta_{j_1}(c_n)\cdots\delta_{j_2 - 1}(c_n)}{\mu^{j_2 - j_1}}\right|<K \quad\text{with}\quad K:=K_5 K_0^{k_0+n_0-1}.
\]
The lemma follows since for $1\leq j\leq m_n-1$, 
\[\Delta_j = \delta_{1}(c_n)\cdots\delta_{j}(c_n)\and\Delta_{-j} = \delta_{m_n-j}(c_n)\cdots\delta_{m_n-1}(c_n).\qedhere\]
\end{proof}

According to Lemma \ref{lemmabelow}, the coefficients of 
\[
P_n(z):=\frac{\chi^2_{f_{c_n}}(z/\mu)}{\chi^2_{f_{c_n}}(0)}=1+\frac{\Delta_{-1}}{\mu}z+\cdots+\frac{\Delta_{-(n-1)}}{\mu^{n-1}}z^{n-1}
\]
and
\[
Q_n(z):=\frac{P_n(z)}{\Delta_{-(n-1)}z^{n-1}} = 1+\frac{\mu}{\Delta_{1}}\frac{1}{z}+\cdots+\frac{\mu^{n-1}}{\Delta_{n-1}}\frac{1}{z^{n-1}}
\]
are uniformly bounded. In particular, the sequence $(P_n)$ is normal in $\Disk$, and the sequence $(Q_n)$ is normal outside $\overline\Disk$. In Lemma \ref{equilemma} below, we prove that as $n\to \infty$, the roots of $P_n$ equidistribute on the unit circle, so the roots of $\chi_{f_{c_n}}$ equidistribute on the circle $\bigl\{|\lambda| = 1/|\mu|\bigr\}$. 
\end{proof}

\begin{lemma}\label{equilemma}
Let 
\[
P_n=1+\cdots+c_nz^{d_n}\in \C[z] \and Q_n=\frac{P_n}{c_nz^{d_n}}=1+\cdots+\frac{1}{c_nz^{d_n}}\in \C[1/z]. 
\]
If 
\begin{itemize}
\item the sequence $(d_n)$ tends to $\infty$ as $n\to \infty$,  
\item the sequence $(P_n)$ is normal in the unit disk $\Disk$,  and 
\item the sequence $(Q_n)$ is normal in $\C\sm \overline\Disk$,
\end{itemize}
then as $n\to \infty$, the roots of $P_n$ equidistribute on the unit circle $S^1$. 
\end{lemma}

\begin{proof}
Extracting a subsequence if necessary, we may assume that the sequence $\m_{P_n}$ converges to a probability measure $\m$ on $\Chat$. It is enough to show that $\m$ coincides with the normalized Lebesgue measure on $S^1$. We first show that the support of $\m$ is contained in $S^1$. We then show that its Fourier coefficients all vanish except the constant coefficient. 

Extracting a further subsequence if necessary, we may assume that 
\begin{itemize}
\item the sequence $(P_n)$ converges locally uniformly in $\Disk$ to a holomorphic map $\phi$, and 
\item the sequence $(Q_n)$ converges locally uniformly in $\Chat\sm \overline \Disk$ to a holomorphic map $\psi$. 
\end{itemize}
Since $P_n(0)=Q_n(\infty)=1$, the limits satisfy $\phi(0)=\psi(\infty)=1$. As a consequence, the zeros of $P_n$ stay bounded away from $ 0$ and $\infty$. In addition, $\phi$ and $\psi$ do not identically vanish, so their zeros  are isolated, and within any compact subset of $\C\sm S^1$, the number of zeros of $P_n$ (counting multiplicities) is uniformly bounded. This shows that the support of $\m$ is contained in $S^1$. 

Now choose $r<1<R$ so that all the roots of $P_n$ remain in the annulus $A:=\{r<|z|<R\}$. For $k\in\Z$, let $m_k$ be the Fourier coefficient
\[
m_k:=\int_{S^1} z^k\  \d\m = \int_A z^k\ \d\m = \lim_{n\to\infty} \int_A z^k \ \d\m_{P_n}
\]
By the Residue Theorem, if $k>0$, then 
\begin{align*}
d_n\cdot  \int_A z^k \ \d\m_{P_n}=\int_{|z|=R} z^k \frac{P_n'(z)}{P_n(z)}\  \dz 
&=\int_{|z|=R} z^k\cdot \left(\frac{d_n}{z}+\frac{Q'_n(z)}{Q_n(z)}\right)\  \dz\\
&=\int_{|z|=R} z^k\frac{Q'_n(z)}{Q_n(z)}\ \dz\\
& \underset{n\to\infty}\longrightarrow \int_{|z|=R} z^k\frac{\psi'(z)}{\psi(z)}\ \dz.
\end{align*}
Similarly, if $k<0$, then 
\begin{align*}
d_n\cdot \int_A z^k \d\m_{P_n}=\int_{|z|=r} z^k\frac{P_n'(z)}{P_n(z)}\  \dz 
\underset{n\to\infty}\longrightarrow \int_{|z|=r} z^k\frac{\phi'(z)}{\phi(z)}\ \dz.
\end{align*}
In both cases, the limit is finite and since $d_n\to \infty$, we deduce that $m_k=0$. 
\end{proof}

To complete the proof of Theorem \ref{theo:4}, it is enough to show that the set of $\mu$ such that there exists a $c_0\in \C$ for which the critical point of $f_{c_0}$ is preperiodic to a fixed point of multiplier $\mu$, is dense in $[1,1/r_D]$. This follows from Lemmas \ref{lem:mcmmult} and \ref{lem:bif} below. 

\begin{lemma}\label{lem:mcmmult}
For each $t\in [1,1/r_D]$, there is a parameter $c\in \partial \M_D$ such that $f_c$ has a fixed point with multiplier of modulus $t$. 
\end{lemma}

\begin{proof}
The boundary of $\M_D$ is connected. As $c$ varies in the boundary of $\M_D$, the multipliers of fixed points vary continuously. Thus, it suffices to show that $\partial \M_D$ contains a parameter $c_0$ for which $f_{c_0}$ has a fixed point with multiplier of modulus $1$, and a parameter $c_1$ for which $f_{c_1}$ has a fixed point of modulus $1/r_D$. 

Note that $f_c(\beta) = \beta$ and $f'_c(\beta) = \mu$ if and only if
\begin{equation}\label{eq:cz}
c^{D-1} = \frac{\mu}{D}\left(1-\frac{\mu}{D}\right)^{D-1}\and \beta = \frac{c}{1-\mu/D}  .
\end{equation}
First, when $c^{D-1} = \frac{1}{D}\left(1-\frac{1}{D}\right)^{D-1}$, then $f_c$ has a fixed point of multiplier $1$. The corresponding parameters $c$ belong to the boundary of $\M_D$.
Second, let $\omega$ be a $D$-th root of unity closest to $-1$. If $D$ is even, we have $\omega=-1$. If $D$ is odd, we have $\omega = \exp\left(±\pi {\rm i}\frac{D-1}{D}\right)$. 
Note that 
\[|1-\omega|^2 =\left(1+\cos \frac{\pi}{D}\right)^2 + \sin^2\frac{\pi}{D} = 2 + 2 \cos \frac{\pi}{D} = 4\cos^2\frac{\pi}{2D}.\]
In both cases,
\[|1-\omega| = \frac{1}{D r_D}.\]
Set $\mu:=D(1-\omega)$, choose $c$ and define $\beta$ so that Equation \eqref{eq:cz} holds. Then, $f_c$ has a fixed point at $\beta$ with multiplier $\mu$ of modulus $1/r_D$. In addition, 
\[f_c(0) = c = \left(1-\frac{\mu}{D}\right)\beta = \omega \beta,\quad \text{so that}\quad f_c^{\circ 2}(0) = \beta.\] 
Thus, $c$ is a {\em Misiurewicz parameter}; that is, $f_c$ is postcritically finite with $c\in \partial \M_D$, as required. 
\end{proof}

\begin{lemma}\label{lem:bif}
Let $c_0\in \partial \M_D$, and let $\beta_0$ be a repelling fixed point of $f_{c_0}$ of multiplier $\mu_0$. Then, there exists a sequence of parameters $c_n$ converging to $c_0$ such that $f_{c_n}$ has a fixed point $\beta_n\in {\cal P}_{f_{c_n}}$ converging to $\beta_0$. 
\end{lemma}

\begin{proof}
Since $\beta_0$ is repelling, there is a function $\beta$ defined and holomorphic near $c_0$, such that $f_c\circ \beta_c = \beta_c$. Let $\omega\neq1$ be a $D$-th root of unity. Since $c_0\in \partial \M_D$, the sequence of functions $\zeta_k:=c\mapsto f_c^{\circ k}(0)$ is not normal at $c_0$. It follows from  Montel's Theorem that in any neighborhood of $c_0$, the sequence $(\zeta_k)$ cannot avoid both $\beta$ and $\omega\beta$. When $\zeta_k = \omega\beta$, then $\zeta_{k+1} = \beta$. So there is a sequence of complex numbers $(c_n)$ converging to $c_0$ and a sequence of integers $(k_n)$ tending to $\infty$ such that $\zeta_{k_n}(c_n) = \beta(c_n)=:\beta_n$.  
\end{proof}

\section{Questions for further study}\label{sec:questions} 

\noindent We conclude with some remaining questions. 

\subsection{Periodic unicritical polynomials}

Let $f$ be  a periodic unicritical polynomial of degree $D$. An eigenvalue of $f_*:\cQ(\Chat)\to \cQ(\Chat)$ satisfies $\frac{1}{4D}<|\lambda|<1$ and for $D$ even $\Sigma(D)$ contains the annulus $\bigl\{\frac{1}{2D}\leq |\lambda|\leq 1\bigr\}$. The estimate obtained for $D$ odd is not as good. 

\begin{question}
For $D$ odd, does $\Sigma(D)$ contain the annulus $\bigl\{\frac{1}{2D}\leq |\lambda|\leq 1\bigr\}$?
\end{question}

We shall say that an eigenvalue of $f_*:\cQ(\Chat)\to \cQ(\Chat)$  is a {\em small eigenvalue} if $|\lambda|<\frac{1}{2D}$. 
According to Proposition \ref{prop:largeeig}, a small eigenvalue belongs to $\spec_f$. In particular, if the critical point is periodic of period $m$, there are are at most $m-2$ small eigenvalues. 

\begin{question}
How many small eigenvalues can a periodic unicritical polynomial have?
\end{question}

\subsection{Spectral gap}

Let $f$ be  a postcritically finite rational map. We saw that for periodic unicritical polynomials of degree $D$, the eigenvalues in $\spec_f\sm\{0\}$ and $\Lambda_f\sm \{0\}$ remain uniformly bounded away from $0$. At the same time, the set of periodic unicritical polynomials of degree $D$ is bounded in moduli space. 

\begin{question}
If the conjugacy class of $f$ remains in a compact subset of moduli space, do the sets $\spec_{f}\sm\{0\}$ and $\Lambda_{f}\sm \{0\}$ remain bounded away from $0$?
\end{question}

We saw that the eigenvalues of $f_*:\cQ(\Chat)\to \cQ(\Chat)$ are related to the multipliers of cycles of $f$. These in turn are related to the Lyapunov exponent $L(f)$ of $f$ (which remains bounded on compact subsets of moduli space since it is continuous). 

\begin{question}
Is there a relation between $\exp\bigl(-L(f)\bigr)$ and $\inf\bigl\{|\lambda|~:~\lambda\in \spec_f\sm\{0\}\bigr\}$ or $\inf\bigl\{|\lambda|~:~\lambda\in \Lambda_f\sm\{0\}\bigr\}$?
\end{question}

\subsection{Diagonalizability}

\begin{question}
Let $f$ be postcritically finite. Is $f_*:\qf\to \qf$ diagonalizable? Is $f_*:\cQ(\Chat)\to \cQ(\Chat)$ diagonalizable?
\end{question}

\begin{question}
Let $f$ be a periodic unicritical polynomial. Are the roots of the characteristic polynomial $\chi_f$ simple?
\end{question}

\end{document}